\documentclass[11pt]{article}
\usepackage{authblk}

\usepackage[margin=1in]{geometry}
\usepackage[T1]{fontenc}
\usepackage[utf8]{inputenc}
\usepackage{amsmath,amssymb,amsthm,mathtools,mathrsfs}
\usepackage[colorlinks=true,linkcolor=blue,citecolor=blue,urlcolor=blue]{hyperref}
\usepackage{enumitem}
\usepackage{needspace}

\newtheorem{theorem}{Theorem}[section]
\newtheorem{lemma}[theorem]{Lemma}
\newtheorem{proposition}[theorem]{Proposition}
\newtheorem{corollary}[theorem]{Corollary}
\theoremstyle{definition}
\newtheorem{definition}[theorem]{Definition}
\theoremstyle{remark}
\newtheorem{remark}[theorem]{Remark}

\DeclareMathOperator{\Sur}{Sur}
\DeclareMathOperator{\Aut}{Aut}
\DeclareMathOperator{\End}{End}
\DeclareMathOperator{\Hom}{Hom}
\DeclareMathOperator{\GL}{GL}
\DeclareMathOperator{\Out}{Out}
\DeclareMathOperator{\Inn}{Inn}
\DeclareMathOperator{\Sym}{Sym}
\DeclareMathOperator{\soc}{soc}
\DeclareMathOperator{\Frat}{Frat}

\DeclareMathOperator{\Gen}{Gen}

\newcommand{\F}{\mathbb F}
\newcommand{\Q}{\mathbb Q}
\newcommand{\Z}{\mathbb Z}
\newcommand{\E}{\mathbb E}
\newcommand{\Gp}{\mathcal G}
\newcommand{\calP}{\mathcal P}

\newcommand{\barS}{\overline{\mathcal S}}
\newcommand{\wideF}{\widehat F}
\newcommand{\Prof}{\operatorname{Prof}}

\title{\vspace{-\baselineskip}\LARGE\sffamily\bfseries
High Rank and Multiplicity in Random and Perfect Profinite Groups}

\author[1]{Carlo Pagano\thanks{Department of Mathematics
and Statistics, Concordia University, Quebec H3G 1M8,
Canada,
\href{mailto:carlein90@gmail.com}{carlein90@gmail.com};
Google DeepMind,
\href{mailto:carlopagano@google.com}{carlopagano@google.com}.}}

\author[2]{Mark Shusterman\thanks{The Dr. A. Edward
Friedmann Career Development Chair in Mathematics at the
Faculty of Mathematics and Computer Science, Weizmann
Institute of Science, 234 Herzl Street, Rehovot 76100,
Israel.}}

\affil[1]{Concordia University and Google DeepMind}
\affil[2]{Weizmann Institute of Science}
\date{\today}

\begin{document}
\maketitle

\begin{abstract}
We prove that almost sure topological finite generation holds for a general class of random models of profinite groups. We deduce that in the models introduced by Liu--Wood and Sawin--Wood, one has that finite presentation holds almost surely, with almost sure control of the deficiency in the presentation. In particular this settles questions raised in \cite{LW} and \cite{SawinWood3M}.

Using the same underlying principle, we show that there exists a unique universal $d$-generated perfect profinite group: its finite quotients are precisely the finite $d$-generated perfect groups. This settles Conjecture B1 posed by Nikolov in \cite{NikolovFiniteImages}. We show in addition that this group is projective and admits a profinite presentation with $d$ generators and $d$ relations, and with no fewer relations on $d$ generators.

The main underlying theme is to bring in an insight from crown theory: groups with high rank are always witnessed by a crown with large multiplicity. 

This approach was discovered independently by \emph{ChatGPT5.5 pro} and \emph{Aletheia}, an internal agent at Google DeepMind.
\end{abstract}

\section{Introduction}\label{sec:introduction}
The rich web of analogies and dictionaries between number fields, algebraic curves and $3$-manifolds is pervasive in modern number theory. Recently, these dictionaries have found particularly fertile grounds in the context of random profinite group models and their applications to arithmetic statistics, in a remarkable series of works of Sawin--Wood \cite{SawinWood3M,SawinWoodMoment,SawinWoodDistributions}, building on previous work of Liu--Wood \cite{LW} and Liu--Wood--Zureick-Brown \cite{LWZB}. Such random groups are used to provide heuristic models for the \'etale fundamental group $\pi_1^{\text{ét}}(\text{Spec}(\mathcal{O}_K))$ when $K$ varies in natural families of number fields. Grounding evidence for these heuristics can be provided for curves over finite fields, thanks to the groundbreaking work of Ellenberg--Venkatesh--Westerland \cite{EllenbergVenkateshWesterland} and the recent breakthrough of Landesmann--Levy \cite{LandesmanLevyMoments,LandesmanLevyStability}. The dependency of these heuristics on the base field leads to a plethora of random group models. It is precisely exploring the analogy with $3$-manifolds that, in the hands of Sawin--Wood \cite{SawinWood3M,SawinWoodDistributions}, the literature has arrived at a convincing model for completely general base number/function fields.

Despite the compelling evidence over function fields, very little is known about $\pi_1^{\text{\'et}}(\text{Spec}(\mathcal{O}_K))$, when $K$ is a number field. A difficult conjecture of Shafarevich \cite[p.~167]{ShafarevichNumberFields}, also motivated by the analogy with algebraic curves, postulates that this is a topologically finitely generated group. Hence it is a natural question whether the random group models produce finitely generated groups almost surely. This question has been raised by Liu--Wood \cite{LW} and by Sawin--Wood \cite{SawinWood3M} for their respective models, as we shall recall more precisely below.

Our first main result settles both cases. With an eye towards future applications, we provide a general statement that can be applied to any random group model as soon as it satisfies a certain growth condition for its so-called \emph{moments}. A moment is the average value of $|\text{Epi}_{\text{top.gr.}}(\mathcal{G},H)|$, where $H$ is any given finite group. One of the key insights, ubiquitous in many works in this area and championed in \cite{SawinWood3M,SawinWoodMoment}, is that these moments provide a systematic way to study our random group model (and they are often sufficiently powerful to encode the whole distribution).

Let $(\Omega,\mathcal{B},\mu)$ be a probability space whose points are profinite groups and such that the functions $\mathcal{G} \mapsto |\text{Epi}_{\text{top.gr.}}(\mathcal{G},H)|$ are measurable and there exist constants $\alpha,\delta\geq 0$ with
$$\sup_{H} \frac{1}{|H|^{\alpha} \cdot |H_2(H,\Z)|^{\delta}} \int_{\Omega} |\text{Epi}_{\text{top.gr.}}(\mathcal{G},H)| \, d\mu < \infty,
$$
where $H$ ranges over all finite groups. We call such a space a \emph{decent space} of random groups. For a profinite group $\mathcal{G}$, we denote by $d(\mathcal{G})$ the infimum among all of the cardinals of the subsets of $\mathcal{G}$ that are topologically generating $\mathcal{G}$.
\begin{theorem} \label{thm:finite-generation}
Let $(\Omega,\mathcal{B},\mu)$ be a decent space of profinite groups. Then the set $\{\mathcal{G} \in \Omega :d(\mathcal{G})<\infty\}$ is measurable, i.e. it is an element of $\mathcal{B}$, and furthermore
$$
\mu(\{\mathcal{G} \in\Omega: d(\mathcal{G})<\infty\})=1,
$$
that is, an element $\mathcal{G}$ of $\Omega$ is topologically finitely generated almost surely.
\end{theorem}
Theorem \ref{thm:finite-generation} has a much more precise incarnation in the context of \cite{LW,SawinWood3M}. To explain that, we briefly recall how these random models are constructed.

Let $\wideF_n$ denote the free profinite group on $n$ generators.  For an integer $u$ and $n+u\geq0$, Liu and Wood \cite{LW} consider
$$
X_{u,n}=\wideF_n/\overline{\langle r_1,\ldots,r_{n+u}\rangle},
$$
where the $r_i$ are independent Haar-random elements of $\wideF_n$.  As $n \to \infty$, these groups converge in law to a measure $\mu_u$ on a natural probability space $\calP$ of profinite groups. We write $X_u$ for a random profinite group with law $\mu_u$.

On the topological side, Dunfield and Thurston \cite{DunfieldThurston} considered a random closed oriented $3$-manifold $M_{g,L}$ by gluing two genus-$g$ handlebodies with a random word of length $L$ in the mapping class group.  Sawin and Wood prove that, first as $L\to\infty$ and then as $g\to\infty$, the groups $\widehat{\pi_1(M_{g,L})}$ converge to a random group picked by a probability measure $\mu_{\mathrm{SW}}$ introduced in  \cite{SawinWood3M}.  At finite level $X_{u,n}$ has $n$ generators and $n+u$ random relations, whereas a genus-$g$ Heegaard splitting gives $g$ generators and $g$ relations. Remarkably, this pattern propagates at the infinite level, as our next two results show.

Recall that Liu and Wood asked whether their limiting groups are nevertheless finitely generated and finitely presented \cite[p.3]{LW}; Sawin and Wood asked whether their limiting group is topologically finitely generated \cite[Section 10.3]{SawinWood3M}.

It is now an immediate consequence of Theorem \ref{thm:finite-generation} that both $X_u$ and a $\mu_{\mathrm{SW}}$-random group are topologically finitely generated almost surely. The next two results say that the same holds for finite presentation and, furthermore, the presentation shape used to construct the groups at finite level propagates at infinite level.
\begin{theorem}
\label{thm:SW-finite-presentation}
The locus in $\mathcal{P}$ of groups $\mathcal{G}$ admitting a profinite presentation with $d:=d(\mathcal{G})<\infty$ generators and $d$ relations, and admitting no profinite presentation with fewer than $d$ relations, is Borel. For
$\mu_{\mathrm{SW}}$-almost every $\mathcal{G} \in \mathcal{P}$ the quantity $d$ is finite and there exist $x_1,\ldots,x_d \in \wideF_d$ such that
$$
\mathcal{G} {\sim}_{\textup{top.gr.}} \wideF_d/\overline{\langle x_1,\ldots,x_d\rangle}.
$$
Moreover, almost surely every finite profinite presentation of $\mathcal{G}$ with $n$ generators and $r$ relations satisfies $r\geq n\geq d$. Thus the minimum numbers of generators and relations are both $d$. In particular, a Sawin--Wood random group is topologically finitely presented with probability $1$.
\end{theorem}
Theorem \ref{thm:SW-finite-presentation} settles in particular Sawin--Wood's question in a strong form.

We next turn to the Liu--Wood model, with the following result.

\begin{theorem}\label{thm:finite-presentation}
The locus of finitely presented groups is Borel in the Liu--Wood space $\mathcal{P}$.  Moreover, for every $u \in \Z$ and for $\mu_u$-almost every $\mathcal{G} \in \mathcal{P}$, the integer $d=d(\mathcal{G})$ is finite, one has that $d>-u$ if $u<0$, and there exist
$x_1,\ldots,x_{d+u}\in\wideF_d$ such that
$$
\mathcal{G}\cong \wideF_d/\overline{\langle x_1,\ldots,x_{d+u} \rangle}.
$$
In particular, $X_u$ is topologically finitely presented with probability $1$.
\end{theorem}
The underlying principle of the proof can be used to settle a conjecture of Nikolov: he conjectured that, for each $d>1$, there exists a topologically finitely generated perfect profinite group $\mathcal{G}$ that has every finite perfect $d$-generated group among its quotients \cite[Conjecture B1]{NikolovFiniteImages}. Our next result proves a slightly stronger statement. 
\begin{theorem}\label{thm:nikolov-B1}
Let $d$ be an integer at least $2$. Then there is a topologically $d$-generated profinite group $U_d$ such that
$$
U_d=[U_d,U_d]
$$
as an abstract group and every finite perfect group $H$ with $d(H)\leq d$ is a continuous quotient of $U_d$. Furthermore $U_d$ is unique up to isomorphism, it is projective and admits a profinite presentation with $d$ generators and $d$ relations; no profinite presentation of $U_d$ has fewer than $d$ relations. 
\end{theorem}

\subsection{Method of proof}\label{subsec:method}
We first describe the main ideas behind the proof of Theorem \ref{thm:finite-generation}.

An initial observation dating back to Cohen--Lenstra \cite{CohenLenstra}, i.e. the very source of random group theory applied to number theory, is that highly symmetric objects should occur rarely. A simple example of a finite group with high rank is $\mathbb{F}_p^n$, i.e. a group which has a high multiplicity, which translates into a massive group of symmetries. At the same time for an abelian (or even pro-nilpotent) $\mathcal{G}$ one would be able to witness the event $d(\mathcal{G})>m$ via the events $\mathcal{G} \twoheadrightarrow \mathbb{F}_p^{m+1}$ and the general Cohen--Lenstra-principle incarnates into the inequality
\begin{equation}\label{eq: CL-intro}
\mu(\mathcal{G} \twoheadrightarrow H) \ll\frac{|H|^{\alpha} \cdot |H_2(H,\Z)|^{\delta}}{|\text{Aut}_{\text{gr}}(H)|},
\end{equation}
for each finite group $H$. Before discussing how one actually capitalizes from this (we still need to sum over all primes $p$), we first need to explain how we face the more fundamental task: namely to find a non-abelian generalization of this picture. We use the following principle from \emph{crown theory}. This theory provides a list of atomic groups called \emph{monoliths} and a notion of \emph{crown power} of a monolith. Before describing these notions, we give the principle.

\textbf{Principle:} \emph{Finite groups with large rank are always witnessed by a large crown power of a monolith.}

To make this principle precise, we remark that the example $\mathbb{F}_p^n \rtimes \mathbb{F}_p^{*}$ shows very clearly that a group might manifest no multiplicity at all, \footnote{in the sense of surjecting on some group of the form $H^N$ for some large $N$.} but have still high rank, hence the need for a subtler notion of power. In the language of crown theory, in this example the crown multiplicity is in fact $n$ and this is the crown power of the monolith $\mathbb{F}_p \rtimes \mathbb{F}_p^{*}$. Monoliths come in two flavors: of abelian type and of non-abelian type. The ones of abelian type are a straightforward generalization of the example $\mathbb{F}_p \rtimes \mathbb{F}_p^{*}$: namely these are the groups of the form $V \rtimes K$, where $V$ is an irreducible faithful representation of $K$ over a prime field. The $n$-th crown power of this is $V^n \rtimes K$.

The deep works of Gasch\"utz--Lucchini \cite{GaschutzCrowns,DallaVoltaLucchini,LucchiniGenerators}, Guralnick--Hoffman \cite{GuralnickHoffman} and Lucchini--Thakkar \cite{LucchiniThakkar} provide a powerful generalization of the abelian case: for $m\geq2$ the event $\{d(\mathcal{G})> m\}$ is always witnessed by $\mathcal{G}$ surjecting to the $f_L(m)$-th crown power where $f_L(m)$ is a function of $m$ that depends on the monolith $L$ in an explicit fashion. 

Without a sufficiently good control on $f_L(m)$ simply listing all of the monoliths (e.g. listing all of relevant $K$'s above inside $\text{Aut}(V))$) would defeat the gain obtained from large automorphisms and make the bound one gets from equation (\ref{eq: CL-intro}) vacuous. That control is provided in the abelian case by Guralnick--Hoffman and in the non-abelian case by Lucchini--Thakkar. Injecting these results into (\ref{eq: CL-intro}) and capitalizing on the obvious automorphisms coming from large multiplicity yields Theorem \ref{thm:finite-generation} upon summing the contributions along all monoliths. A slightly subtler detail here is that at the phase transition $\delta=2$ one needs to switch from automorphisms of a large crown power to surjections of a large crown power to a smaller one and a corresponding version of equation \eqref{eq: CL-intro}. We emphasize that this part of the proof, devoted to control the quantity $f_L(m)$, relies indirectly, yet heavily, on the classification of finite simple groups, as the currently known proofs of these results rely on that. 

To pass from Theorem \ref{thm:finite-generation} to Theorem \ref{thm:SW-finite-presentation}, we use the criterion of Sawin--Wood \cite[Theorem~1.5]{SawinWood3M} and a relation-module argument giving $d$ generators and at most $d$ relations \cite[Theorems~3.1 and~5.1]{Lubotzky}. The opposite inequality comes from the fact that the maximal abelian pro-$2$ quotient is finite almost surely: a presentation with $n$ generators and $r$ relations gives a quotient of $\Z_2^n$ by $r$ vectors, and finiteness forces $r\geq n$. For Theorem \ref{thm:finite-presentation}, convergence of the finite Liu--Wood models supplies presentations in each finite pro-$S$-category. A relation-module argument reduces the number of generators while preserving the difference between the numbers of relations and generators \cite[Lemma~13.1]{LW}. This reduction requires at least one relator; for $u<0$, we prove that $d(\mathcal{G})>-u$ almost surely, which supplies this condition. A compactness argument then gives the global presentation.

Finally in the proof of Theorem \ref{thm:nikolov-B1} we run the principle above in the opposite direction. The group we build must cover every finite perfect $d$-generated group, so it contains chief series of unbounded length, and the whole point is to stack these chief factors so that no monolith $L$ reaches crown multiplicity $f_L(d)$, the critical multiplicity at which the rank would exceed $d$. This process happens in Lemma \ref{lem:amalgamation} and we explain it in a self-contained manner, however the steps in the proof naturally stem out of crown theory and we expand that in Remark \ref{remark: interpret}. 
\subsection{Further results}
For the Liu--Wood measure we can go beyond almost surely finite generation and provide precise asymptotic control on the distribution of $\mathcal{G} \mapsto d(\mathcal{G})$.
\begin{theorem}
\label{thm:generator-rank-first-order}
For every fixed $u\in\Z$,
$$
 \mu_u\{\mathcal{G}\in\mathcal{P}:d(\mathcal{G})=r\}
 \sim
 \frac{2^{-r(r+u)}}
 {\displaystyle\prod_{j=1}^{\infty}(1-2^{-j})}
 \qquad(r\to\infty).
$$
More precisely, if
$\eta_q(k)=\prod_{j=1}^k(1-q^{-j})$ and
$\eta_q(\infty)=\prod_{j\geq1}(1-q^{-j})$, then
$$
 \mu_u\{\mathcal{G}\in\calP:d(\mathcal{G})=r\}
 =
 \frac{\eta_2(\infty)}{\eta_2(r)\eta_2(r+u)}
 2^{-r(r+u)}
 +O_u\bigl(3^{-r^2+C_ur}\bigr)
$$
for all sufficiently large $r$, where $C_u$ depends only on $u$.
\end{theorem}
Theorem \ref{thm:generator-rank-first-order} can be interpreted as saying that the crown powers of the monolith $\mathbb{F}_2$ are asymptotically capturing the behavior of the function $d(-)$. In fact, we believe that Theorem \ref{thm:generator-rank-first-order} is only the first order approximation of the full story. There should be a full asymptotic expansion labeled by monoliths suitably ordered each coming at the appropriate time with an explicit term. It would be very interesting to explore this story in full.

Recall that a profinite group $\mathcal{G}$ is \emph{positively finitely generated} (PFG) if, for some integer $d\geq1$, the tuples that generate $\mathcal{G}$ topologically form a subset of positive Haar measure in $\mathcal{G}^d$ \cite{MannShalev}. It has \emph{uniformly bounded exponential representation growth} (UBERG) if there exists a constant $c>0$ such that, for every prime $p$ and every integer $n\geq1$, the number of isomorphism classes of continuous irreducible $n$-dimensional representations of $\mathcal{G}$ over $\mathbb{F}_p$ is at most $p^{cn}$ \cite[Definition 5.1]{KionkeVannacci}. Positive finite generation implies UBERG \cite[Corollary 1.11]{CorobCookVannacci}. Our next result determines exactly when these properties hold in the Liu--Wood model.
\begin{theorem}\label{thm:intro-PFG}
Let $u\in\Z$, and let $\mathcal{G}$ have law $\mu_u$.
If $u\geq-1$, then almost surely $\mathcal{G}$ is PFG
and has UBERG.
If $u\leq-2$, then almost surely $\mathcal{G}$ is not PFG
and does not have UBERG.
\end{theorem}
See Theorem \ref{thm:LW-PFG} for a complete and slightly more refined statement, controlling some crown multiplicities of non-abelian types almost surely.

Finally, we remark that the original random group model introduced in Liu--Wood \cite{LW} needs to accomodate extra-structure when one considers families of $\Gamma$-extensions of $\Q$, mostly coming from the $\Gamma$-action, as carried out by Liu--Wood--Zureick-Brown \cite{LWZB}. One can see that Theorem \ref{thm:finite-generation} applies easily to that case. This is done in Section \ref{subsec:equivariant-transfer}. 
\subsection{Layout of the paper}
Section \ref{sec:witnesses} develops the basics and the notion of $m$-witnesses. Section \ref{sec:crowns} implements crown theory and gives a very efficient set of $m$-witnesses.
Section \ref{sec:estimates} carries out the resulting estimates, proves Theorem \ref{thm:finite-generation}, and applies it to both the limiting Liu--Wood and Sawin--Wood measures.  Section \ref{sec:finite-presentation} proves Theorems \ref{thm:SW-finite-presentation} and
\ref{thm:finite-presentation}.  Section \ref{sec:generator-rank-asymptotics} proves
Theorem~\ref{thm:generator-rank-first-order}.  Section \ref{sec:perfect-universal} proves Theorem \ref{thm:nikolov-B1}, that is Nikolov's Conjecture B1, together with the uniqueness, projectivity and presentation statements. Section \ref{sec:further-applications} proves Theorem \ref{thm:intro-PFG} and gives further applications of Theorem \ref{thm:finite-generation}.
\subsection{Human--AI collaboration}
The idea of using crown theory to transfer high generator rank into a high automorphism cost was discovered first by \emph{GPT5.5 pro} in two separate conversations with the authors (respectively about the solvable regime, i.e. abelian type crowns, and non-solvable regime, i.e. non-abelian type crowns). The entire proof of Theorem \ref{thm:finite-generation} and Theorem \ref{thm:finite-presentation} was subsequently and independently discovered also by \emph{Aletheia} \cite{aletheia}, an internal agent at Google DeepMind. Autonomous proofs of Theorem \ref{thm:finite-generation} and Theorem \ref{thm:finite-presentation} were also obtained using Google's \emph{Colosseum} harness \cite{colosseum} and \emph{Cogentic framework} from Google Research. 

While digesting the ideas of the proof, the authors have then explored natural further directions and applications, which culminated in the other results presented above. This exploration has been done in a collaborative fashion with \emph{GPT5.6 Sol pro}, \emph{Fable} as well as with internal agents at Google DeepMind. We now explain this exploration in greater detail, outlining what happened for each of the other results:
\begin{enumerate}
\item Theorem \ref{thm:SW-finite-presentation}, Theorem \ref{thm:generator-rank-first-order} and Theorem \ref{thm:intro-PFG} were discovered by the authors. 
\item Theorem \ref{thm:nikolov-B1} was discovered as follows. The authors wanted to explore to what extent Theorem \ref{thm:finite-generation} could be used from the viewpoint of the \emph{probabilistic method} to construct interesting profinite groups. They prompted AI-agents from all the models above to explore possibilities. After scanning a long list of possible applications provided by AI, the authors selected to explore Conjecture B1 of Nikolov, where \emph{GPT5.6 Sol pro} discovered a proof. 
\end{enumerate}
Once the process above settled, the authors have used \emph{Astra} to draft this work in a back-and-forth fashion as follows. The first version was produced by AI following the authors' instructions. They then rewrote, or in some places asked \emph{Astra} to rewrite, or occasionally both, the entirety of the text, which was subsequently proofread (by both the authors and \emph{Astra}) and edited several times until it reached a status that satisfied them.

\subsection*{Acknowledgments}
The authors are grateful to Yuan Liu and Melanie Wood for past conversations about this problem. They thank Tony Feng and Jack Miller for comments on a preliminary version of this work that led to an improvement of the exposition. CP is grateful to the members of the \emph{Superhuman reasoning team} at Google DeepMind for reproducing Theorem \ref{thm:finite-generation} and Theorem \ref{thm:finite-presentation} autonomously as well as for several conversations about artificial intelligence and mathematics. CP also thanks David P.~Woodruff and the \emph{Cogentic team} for sharing the
autonomous proofs of Theorems~\ref{thm:finite-generation}
and~\ref{thm:finite-presentation} obtained using \emph{Colosseum} and \emph{Cogentic framework} respectively. MS is co-funded by the European Union (ERC, Function Fields, 101161909).

\section{\texorpdfstring{$m$}{m}-witnesses}\label{sec:witnesses}

For a profinite group $\mathcal{G}$, let $d(\mathcal{G})$ denote the least cardinality of a finite
topological generating set, with $d(\mathcal{G})=\infty$ if no such set exists.  All
epimorphisms from profinite groups to finite groups will be understood to be
continuous.

\begin{definition}\label{def:decent-space}
Let $(\Omega,\mathcal{B},\mu)$ be a probability space whose points are profinite
groups up to isomorphism.  For every finite group $H$, set
$$
\nu_H(\mathcal{G})=|\Sur(\mathcal{G},H)| \in\Z_{\geq0} \cup \{\infty\},
$$
where any infinite cardinality is recorded as $\infty$.  We say that
$(\Omega,\mathcal{B},\mu)$ is \emph{decent} if $\nu_H$ is measurable for every finite group $H$ and if
there are fixed constants $M\geq1$ and $\alpha,\delta\geq0$ such that for all finite groups $H$
\begin{equation}\label{eq:Schur-moment-growth}
\int\nu_H(\mathcal{G})d\mu \leq M|H|^\alpha|H_2(H,\Z)|^\delta.
\end{equation}
\end{definition}

\begin{remark}\label{rem:quotient-event}
For a finite group $H$, the event
$$
E(H):=\{\mathcal{G} \in \Omega: \mathcal{G} \twoheadrightarrow H\}=\{\mathcal{G} \in \Omega:\nu_H(\mathcal{G})>0\}
$$
is measurable in every decent probability space.  Moreover,
postcomposition defines a free action of $\text{Aut}(H)$ on $\text{Sur}(\mathcal{G},H)$.
Consequently,
$$
\mathbf 1_{E(H)}(\mathcal{G}) \leq \frac{\nu_H(\mathcal{G})}{|\Aut(H)|}.
$$
In particular, in a decent probability space
\begin{equation}\label{eq:quotient-probability}
\mu(E(H)) \leq \frac{\E_{\mu}[\nu_H]}{|\Aut(H)|} \leq M \cdot \frac{|H|^\alpha \cdot |H_2(H,\Z)|^\delta}{|\Aut(H)|}.
\end{equation}
\end{remark}

\begin{proposition}\label{prop:finite-quotient-detects-generation}
Let $\mathcal{G}$ be a profinite group and let $m \geq 0$ be an integer. Then $d(\mathcal{G})>m$ if and only if
$\mathcal{G}$ has a finite continuous quotient $H$ such that $d(H)>m$.
\end{proposition}

\begin{proof}
One implication is immediate, since the image of a topologically generating set generates every finite continuous quotient. For the converse one has merely to select in a compatible fashion a set of generators of size $m$ among the ones available at every finite level. This follows at once from the fact that an inverse limit of finite non-empty sets is itself non-empty, which can be easily seen by a compactness argument.
\end{proof}

For $m\geq0$, write
$$
E_m:=\{\mathcal{G} \in \Omega: d(\mathcal{G})>m\}.
$$

\begin{proposition}\label{prop:Em-measurable}
If $(\Omega,\mathcal{B},\mu)$ is decent, then $E_m$ is measurable.
\end{proposition}

\begin{proof}
There are only countably many isomorphism classes of finite groups of rank bigger than $m$ and Proposition
\ref{prop:finite-quotient-detects-generation} gives
$$
E_m=\bigcup_{H: d(H)>m}E(H).
$$
Each event on the right is measurable by Remark \ref{rem:quotient-event}, so
$E_m$ is measurable.
\end{proof}

Proposition \ref{prop:finite-quotient-detects-generation} suggests detecting $E_m$ using a smaller collection of carefully chosen finite quotients.

\begin{definition}\label{def:m-witnesses}
A countable collection $\{H_i\}_{i\in I}$ of finite groups is a
\emph{set of $m$-witnesses} if, for every profinite group $\mathcal{G}$,
$$
d(\mathcal{G})>m \quad\Longleftrightarrow\quad \mathcal{G}\twoheadrightarrow H_i \text{ for some }i\in I.
$$
Equivalently, in every probability space of profinite groups,
$$
E_m=\bigcup_{i\in I}E(H_i).
$$
\end{definition}

For every finite group $H$, the event $E(H)$ occurring here is measurable in every decent probability space by Remark \ref{rem:quotient-event}.

\begin{proposition}\label{prop:witness-union-bound}
Let $\{H_i\}_{i\in I}$ be a set of $m$-witnesses, and let
$(\Omega,\mathcal{B},\mu)$ be a decent probability space of profinite groups.  Then
$$
\mu(E_m)\leq\sum_{i\in I}\mu(E(H_i)).
$$
If the space is decent with constants
$M,\alpha,\delta$, then
\begin{equation}\label{eq:witness-moment-bound}
\mu(E_m) \leq M \cdot \sum_{i\in I} \frac{|H_i|^\alpha \cdot |H_2(H_i,\Z)|^\delta}{|\Aut(H_i)|}.
\end{equation}
\end{proposition}

\begin{proof}
The first assertion follows from Definition \ref{def:m-witnesses} and
countable subadditivity.  The second follows by applying
\eqref{eq:quotient-probability} to each $H_i$.
\end{proof}

\begin{remark}\label{rem:all-finite-witnesses}
Proposition \ref{prop:finite-quotient-detects-generation} shows that the
collection of all finite groups $H$ with $d(H)>m$, with one representative
from each isomorphism class, is a set of $m$-witnesses.  Our next goal is to
replace this collection by a substantially more efficient one for which the
series in \eqref{eq:witness-moment-bound} can be estimated.
\end{remark}

\section{Crown theory}\label{sec:crowns}

\subsection{Abstract and concrete monoliths}

\begin{definition}\label{def:monolithic}
A finite group $L$ is \emph{monolithic} if it has a unique minimal normal
subgroup; this subgroup is denoted by $A=\soc(L)$ and called the
\emph{monolith} of $L$.  The group is \emph{monolithic primitive} if
$A\not\leq\Frat(L)$.

For $t\geq1$, the $t$-th \emph{crown-based power} of a monolithic group $L$
is
$$
L_t:=\{(\ell_1,\ldots,\ell_t)\in L^t: \ell_1A=\cdots=\ell_tA\}=L\times_{L/A}\cdots\times_{L/A}L.
$$
Thus $L_t/A^t\cong L/A$, and projection onto any $s\leq t$ coordinates is an epimorphism $L_t\twoheadrightarrow L_s$.
\end{definition}

\begin{definition}\label{def:concrete-monoliths}
A \emph{concrete monolithic datum} is one of the following.
\begin{enumerate}[label=\textup{(\Alph*)}]
\item A prime $p$, a nonzero finite-dimensional $\F_p$-vector space $V$, and
a subgroup $K\leq\GL(V)$ acting faithfully and irreducibly.  We put
$$
L=V\rtimes K
$$
and call $L$ a monolithic group of \emph{abelian type}.  The case $K=1$
forces $V\cong \mathbb{F}_p$.
\item A finite nonabelian simple group $S$, an integer $\rho\geq1$, and a
subgroup
$$
K\leq\Out(S^\rho) \cong\Out(S)\wr\Sym_\rho
$$
whose permutation action on the $\rho$ simple factors is transitive.  If
$$
\pi:\Aut(S^\rho)\twoheadrightarrow\Out(S^\rho)
$$
is the quotient map, we put $L=\pi^{-1}(K)$ and call $L$ a monolithic group of \emph{nonabelian type}.  
\end{enumerate}
\end{definition}
We remark that in the nonabelian case the resulting group extension need not to be a split one. 
\begin{proposition}\label{prop:concrete-monoliths}
A finite group is monolithic primitive if and only if it
is isomorphic to a group arising from construction
\textup{(A)} or \textup{(B)} in
Definition~\ref{def:concrete-monoliths}.
\end{proposition}

\begin{proof}
Let $L$ be monolithic primitive, with unique minimal
normal subgroup $A$.

Suppose first that $A$ is abelian. Then $A$ is elementary
abelian. Since $A\not\leq\Frat(L)$, there is a maximal
subgroup $U$ of $L$ not containing $A$. Thus $L=AU$.
The subgroup $A\cap U$ is normalized by $U$ and centralized
by $A$, hence is normal in $L$. Minimality of $A$ gives
$A\cap U=1$, so $L=A\rtimes U$. The action of $U$ on $A$
is irreducible by minimality of $A$. Its kernel is
normalized by $U$ and centralized by $A$, hence is normal
in $L$. Since it intersects $A$ trivially, it must be
trivial. Thus the action is faithful, giving
construction \textup{(A)}.

Suppose now that $A$ is nonabelian. Then $A\cong S^\rho$
for a finite nonabelian simple group $S$ and an integer
$\rho\geq1$. The subgroup $C_L(A)$ is normal in $L$ and
satisfies $C_L(A)\cap A=Z(A)=1$, so $C_L(A)=1$.
Conjugation therefore embeds $L$ in $\Aut(A)$, with
$A$ identified with $\Inn(A)$. Since this image contains
$\Inn(A)$, it is the full inverse image of a subgroup
$K\leq\Out(A)$. The product of the simple factors in
any $K$-orbit is normal in $L$, so minimality of $A$
forces this action to be transitive. This gives
construction \textup{(B)}.

Conversely, let $L$ arise from either construction,
and let $A$ denote $V$ in \textup{(A)} and
$\Inn(S^\rho)\cong S^\rho$ in \textup{(B)}.
Irreducibility in \textup{(A)} and transitivity in
\textup{(B)} show that $A$ is minimal normal in $L$.
Moreover, $C_L(A)=A$ in \textup{(A)} and $C_L(A)=1$
in \textup{(B)}. If $N\trianglelefteq L$ satisfies
$N\cap A=1$, then $[N,A]\leq N\cap A=1$; these
centralizer identities therefore imply $N=1$.
Thus $A$ is the unique minimal normal subgroup.

In \textup{(A)}, a maximal subgroup containing the
complement $K$ cannot contain $A$, so
$A\not\leq\Frat(L)$. In \textup{(B)}, $A$ is not
nilpotent, whereas $\Frat(L)$ is nilpotent, so again
$A\not\leq\Frat(L)$. Hence $L$ is monolithic primitive.
\end{proof}
The data in these two descriptions can be recovered
intrinsically from $L$. Let $A$ be its unique minimal
normal subgroup. Conjugation identifies $L/A$ with
$K\leq\Out(A)$. In type \textup{(A)}, we recover $V=A$,
with its natural $\F_p$-vector space structure, and
$\Out(A)=\GL(V)$. In type \textup{(B)}, the simple group
$S$, up to isomorphism, and the integer $\rho$ are
determined by $A\cong S^\rho$.
\begin{remark}
The word \emph{primitive} in Proposition \ref{prop:concrete-monoliths} is essential. A monolithic group whose monolith lies in the Frattini subgroup never occurs in the groups in lists \textup{(A)} and \textup{(B)}. 
\end{remark}

\subsection{The first bad crown power}

Let $L$ be monolithic primitive.  The sequence $d(L_t)$ is nondecreasing and
unbounded \cite{DallaVoltaLucchini,LucchiniGenerators}.  For
$m\geq d(L)$ an integer, define
\begin{equation}\label{eq:def-fLm}
f_L(m):=\min\{t\geq1:d(L_t)>m\}.
\end{equation}
Thus $L_{f_L(m)}$ is the first crown-based power of $L$ which is not $m$-generated.

We record explicit descriptions of this function.  Suppose first that
$L=V\rtimes K$ is of abelian type.  Put
$$
D=\End_K(V)\cong\F_q,\qquad b=\dim_DV,\qquad s=\dim_DH^1(K,V),
$$
and let
$$
\theta=\begin{cases} 0,&\text{if $V$ is the trivial $K$-module},\\ 1,&\text{otherwise}.\end{cases}
$$
The Gasch\"utz--Dalla Volta--Lucchini formula gives
\begin{equation}\label{eq:abelian-generator-formula}
d(L_t)=\max\left\{d(K),\, \theta+\left\lceil\frac{t+s}{b}\right\rceil \right\}.
\end{equation}
Consequently, whenever $d(L)\leq m$, we have that
\begin{equation}\label{eq:abelian-fLm}
f_L(m)=b(m-\theta)-s+1.
\end{equation}
See \cite[Theorem~2.7]{DallaVoltaLucchini}. Formula
\eqref{eq:abelian-generator-formula} also includes the cyclic monolith $L=\mathbb{F}_p$, for which $b=1$, $s=\theta=0$, and $f_L(m)=m+1$.

For a finite group $G$, write
$$
\Gen_m(G):=\{(g_1,\ldots,g_m)\in G^m:\langle g_1,\ldots,g_m\rangle=G\}.
$$

We now give an explicit formula in the
nonabelian case.  Let $A=\soc(L)$, put $K=L/A$, and write
$$\Aut_1(L):=\{\sigma\in\Aut(L): \sigma(\ell)A=\ell A\text{ for every }\ell\in L\}.
$$
For every $m\geq d(L)$, Dalla Volta--Lucchini's conditional generation formula gives the following expression.  Since $K=L/A$ is a quotient of $L$, one has $d(K)\leq d(L)\leq m$; in particular, $\Gen_m(K)\neq\varnothing$.  Thus
\begin{equation}\label{eq:nonabelian-fLm}
f_L(m)=1+\frac{|\Gen_m(L)|}{|\Aut_1(L)|\,|\Gen_m(K)|}.
\end{equation}
The quotient in \eqref{eq:nonabelian-fLm} is an integer.  It counts the $\Aut_1(L)$-orbits of generating lifts to $L$ of a fixed generating $m$-tuple of $K$; the number of lifts is independent of the chosen tuple, and the action is free because a generating tuple determines an automorphism.  Indeed, an $m$-tuple in $L_t$ generates $L_t$ precisely
when its $t$ coordinate tuples generate $L$ and determine distinct such orbits; see \cite[Theorem~2.7]{DallaVoltaLucchini}.  In particular, if
$L=S$ is nonabelian simple, then
\begin{equation}\label{eq:simple-fLm}
f_S(m)=1+\frac{|\Gen_m(S)|}{|\Aut(S)|}.
\end{equation}
Thus, apart from the generation probability and the outer automorphisms, $f_S(m)$ has size $|S|^{m-1}$.
\begin{theorem}\label{thm:crown-witnesses}
Let $m\geq2$, and let $H$ be a finite group with $d(H)>m$.  Then there is a
concrete monolithic primitive group $L$, with $d(L)\leq m$, such that
$$
H \twoheadrightarrow L_{f_L(m)}.
$$
Moreover, for every monolithic primitive group $L$ with $d(L)\leq m$, the following bounds hold.
If $L=V\rtimes K$ is of abelian type and
$b=\dim_{\End_K(V)}V$, then
\begin{equation}\label{eq:abelian-f-gap}
f_L(m)>b(m-2).
\end{equation}
If $A=\soc(L)$ is nonabelian, then
\begin{equation}\label{eq:nonabelian-f-gap}
f_L(m)>|A|^{m-8/5}>|A|^{m-2}.
\end{equation}
\end{theorem}

\begin{proof}
Choose a quotient $H_0$ of $H$ of least order subject to
$d(H_0)>m$. Then
$$
        d(H_0/N)\leq m
        \qquad
        (1\neq N\trianglelefteq H_0).
$$
The Dalla Volta--Lucchini reduction
\cite[Theorem~1.4]{DallaVoltaLucchini} gives
$$
H_0\cong L_t
$$
for a monolithic primitive group $L$ whose monolith is either nonabelian or possessing a complement in $L$, and for some $t\geq2$.  Proposition \ref{prop:concrete-monoliths} puts $L$ in one of the two concrete forms, namely case $(A)$ or case $(B)$ of Definition \ref{def:concrete-monoliths}. Projection onto any $s<t$ coordinates gives a quotient map $L_t\twoheadrightarrow L_s$ with non-trivial kernel.  Minimality of $H_0$ therefore gives $d(L_s)\leq m$ for every $s<t$.  Since $d(L_t)>m$, we have $t=f_L(m)$.  In particular, $L=L_1$ is a proper quotient of $H_0$, so $d(L)\leq m$.

For the bounds, let $L$ now be any monolithic primitive group with $d(L)\leq m$.
Suppose first that $L=V\rtimes K$. If $V$ is the trivial $K$-module, faithfulness forces $K=1$, and \eqref{eq:abelian-fLm} gives $f_L(m)=m+1>b(m-2)$.  If $V$ is a nontrivial $K$-module, the Guralnick--Hoffman theorem gives
$$\dim_{\F_p}H^1(K,V) \leq\frac12\dim_{\F_p}V \qquad\text{\cite[Theorem 1]{GuralnickHoffman}}.
$$
The field $D$ acts on $H^1(K,V)$, so $s\leq b/2$. Since $\theta=1$, \eqref{eq:abelian-fLm} yields
$$ f_L(m)=b(m-1)-s+1 \geq b(m-3/2)+1>b(m-2).
$$

Suppose now that $A$ is nonabelian, and let
$N\cong A$ be one coordinate subgroup of
$A^{f_L(m)}=\soc(L_{f_L(m)})$.  Then
$$
        L_{f_L(m)}/N\cong L_{f_L(m)-1}
$$
is $m$-generated.  If $f_L(m)\leq |A|^{m-8/5}$, then
$$
\left\lceil \frac85+\log_{|A|}f_L(m)  \right\rceil \leq m.
$$
The Lucchini--Thakkar lifting theorem \cite[Proposition 6]{LucchiniThakkar} would then lift an $m$-element
generating tuple modulo $N$ to an $m$-element generating tuple of $L_{f_L(m)}$, a contradiction. This proves
\eqref{eq:nonabelian-f-gap}.
\end{proof}

\begin{corollary}\label{cor:crown-witness-bound}
Let $m\geq2$, and let $\mathcal L_m$ contain one representative from each isomorphism class of concrete monolithic primitive groups $L$ with $d(L)\leq m$.  Then
$$
\{L_{f_L(m)}:L\in\mathcal L_m\}
$$
is a set of $m$-witnesses.  If $(\Omega,\mathcal{B},\mu)$ is decent with constants $M,\alpha,\delta$, then
\begin{align}
\mu(E_m) &\leq M\sum_{L\in\mathcal L_m} \frac{|L_{f_L(m)}|^\alpha |H_2(L_{f_L(m)},\Z)|^\delta}{|\Aut(L_{f_L(m)})|}\label{eq:exact-crown-bound}\\
&\leq M\sum_{L\in\mathcal L_m}\sum_{n\geq f_L(m)}\frac{|L_n|^\alpha|H_2(L_n,\Z)|^\delta}{|\Aut(L_n)|}.  \label{eq:crown-tail-bound}
\end{align}
\end{corollary}

\begin{proof}
Combining Theorem \ref{thm:crown-witnesses} with Proposition \ref{prop:witness-union-bound} we obtain
\eqref{eq:exact-crown-bound}. Considering that the additional terms are nonnegative gives
\eqref{eq:crown-tail-bound}.
\end{proof}

\section{Estimates}\label{sec:estimates}
Throughout this section, $(\Omega,\mathcal B,\mu)$ is a decent space with constants $M\geq1$ and $\alpha,\delta\geq0$ as in \eqref{eq:Schur-moment-growth}. Put
$$
c_0=\prod_{i=1}^{\infty}(1-2^{-i})>0.
$$

\Needspace{8\baselineskip}
\begin{proposition}\label{prop:crown-automorphisms}
Let $L_t$ be a crown-based power.
\begin{enumerate}[label=\textup{(\roman*)}]
\item Suppose that $L=V\rtimes K$ is of abelian type. Write
$$
D=\End_K(V)\cong\F_q.
$$
Then
$$
|\Aut(L_t)| \geq |\GL_t(q)|=q^{t^2}\prod_{i=1}^t(1-q^{-i}) \geq c_0q^{t^2}.
$$
\item Suppose that $L$ is of nonabelian type, with $A=\soc(L)$. Then
$$
|\Aut(L_t)| \geq |A|^t\cdot t!.
$$
\end{enumerate}
\end{proposition}

\begin{proof}
In the abelian case, $L_t=V^t\rtimes K$. As a $K$-module, $V^t$ may be identified with $V\otimes_DD^t$, where $K$ acts on the first factor. Thus every $\gamma\in\GL_t(D)$ induces the automorphism
$$
(v,k)\longmapsto((1\otimes\gamma)v,k)
$$
of $L_t$. This gives an embedding $\GL_t(q)\hookrightarrow\Aut(L_t)$. Counting ordered bases gives
$$
|\GL_t(q)|=q^{t^2}\prod_{i=1}^t(1-q^{-i}),
$$
and the final inequality follows from $q\geq2$.

Now suppose that $A$ is nonabelian. Since $A$ is a direct product of nonabelian simple groups, $A^t$ is centerless. Conjugation therefore embeds $A^t$ into $\Aut(L_t)$. Coordinate permutations preserve
$$
L_t=\{(\ell_1,\ldots,\ell_t)\in L^t: \ell_1A=\cdots=\ell_tA\}
$$
and give a copy of $\Sym_t$ in $\Aut(L_t)$. The two subgroups intersect trivially, and together they form
\[
A^t\rtimes\Sym_t \leq \Aut(L_t). \qedhere
\]
\end{proof}

\begin{lemma}\label{lem:H2-general-bound}
If $X$ is a finite $d$-generated group, then
$$
|H_2(X,\Z)| \leq |X|^d.
$$
\end{lemma}

\begin{proof}
For a finite abelian group $B$, write $B^\vee=\Hom(B,\Q/\Z)$, with the contragredient action when $B$ is a module. We argue by induction on $|X|$. The case $X=1$ is immediate, so assume $X\neq1$ and $d\geq1$. Let $N$ be a minimal normal subgroup of $X$, and put $Q=X/N$. The Lyndon--Hochschild--Serre homology spectral sequence gives
\begin{equation}\label{eq:LHS-H2}
|H_2(X,\Z)| \leq |H_2(Q,\Z)|\,|H_1(Q,H_1(N,\Z))|\,|H_0(Q,H_2(N,\Z))|.
\end{equation}
Since $d(Q)\leq d$, induction gives
$$
|H_2(Q,\Z)| \leq |Q|^d.
$$
Suppose first that $N$ is abelian. Then $N$ is an irreducible $\F_p[Q]$-module, and $H_1(Q,N)^\vee\cong H^1(Q,N^\vee)$. A $1$-cocycle is determined by its values on a generating $d$-tuple of $Q$. If the action on $N$ is nontrivial, then $(N^\vee)^Q=0$, so the group of coboundaries has order $|N|$. Thus
$$
|H_1(Q,N)| \leq |N|^{d-1}.
$$
Moreover, $H_2(N,\Z)=\wedge^2_{\F_p}N$. The dual of its coinvariants is the space of $Q$-invariant alternating forms on $N$. This space embeds in $\Hom_Q(N,N^\vee)$ and therefore has order at most $|\End_Q(N)|\leq|N|$. If the action on $N$ is trivial, then $N=C_p$, its exterior square vanishes, and the cocycle count gives $|H_1(Q,N)|\leq|N|^d$. In both cases, the last two factors in \eqref{eq:LHS-H2} have product at most $|N|^d$.

Suppose now that $N=S^\rho$ for a nonabelian finite simple group $S$. Then $H_1(N,\Z)=0$, while
$$
H_2(N,\Z) \cong H_2(S,\Z)^\rho.
$$
The bound $|H_2(S,\Z)|\leq|S|$ for nonabelian finite simple groups \cite[pp.~viii--ix, xv--xvi and 74]{Atlas} gives
$$
|H_0(Q,H_2(N,\Z))| \leq |H_2(N,\Z)| \leq |N|\leq|N|^d.
$$
In either case \eqref{eq:LHS-H2} gives
\[
|H_2(X,\Z)| \leq |Q|^d|N|^d=|X|^d. \qedhere
\]
\end{proof}

\begin{proposition}\label{prop:crown-schur}
Let $L_s$ be a crown-based power.
\begin{enumerate}[label=\textup{(\roman*)}]
\item Suppose that $L=V\rtimes K$ is of abelian type. Write
$$
D=\End_K(V)\cong\F_q, \qquad b=\dim_DV.
$$
Then
\begin{equation}\label{eq:abelian-schur}
|H_2(L_s,\Z)| \leq |K|^{d(K)}\cdot q^{s(s+1)/2+bs/2}.
\end{equation}
If $K=1$, then $L=C_p$, $q=p$, $b=1$, and more precisely
$$
|H_2(L_s,\Z)|=p^{\binom s2}.
$$
\item Suppose that $L$ is of nonabelian type. Put
$$
A=\soc(L),\qquad K=L/A,\qquad N=|A|.
$$
Then
\begin{equation}\label{eq:nonabelian-schur}
|H_2(L_s,\Z)| \leq |K|^{d(K)}N^s.
\end{equation}
\end{enumerate}
\end{proposition}

\begin{proof}
In the abelian case, the spectral sequence for
$$
1\longrightarrow V^s\longrightarrow L_s\longrightarrow K\longrightarrow1
$$
gives
$$
|H_2(L_s,\Z)| \leq |H_2(K,\Z)|\,|H_1(K,V^s)|\,|H_0(K,\wedge^2V^s)|.
$$
By Lemma \ref{lem:H2-general-bound},
$$
|H_2(K,\Z)| \leq |K|^{d(K)}.
$$
If $K\neq1$, Pontryagin duality and the Guralnick--Hoffman theorem \cite[Theorem 1]{GuralnickHoffman} applied to $V^\vee$ give
$$
|H_1(K,V)| \leq q^{b/2},
$$
and hence $|H_1(K,V^s)|\leq q^{bs/2}$.

Here tensor and exterior powers of $V$ are taken over $\F_p$, where $p=\operatorname{char}(\F_q)$. Write $V^s=V_1\oplus\cdots\oplus V_s$. Then
$$
\wedge^2V^s=\bigoplus_i\wedge^2V_i \oplus \bigoplus_{i<j}V_i\otimes_{\F_p}V_j.
$$
The dual of the coinvariants of each cross-term is $\Hom_K(V,V^\vee)$, whose order is at most $q$ by Schur's lemma. The same bound holds for a diagonal term, since invariant alternating forms form a subspace of $\Hom_K(V,V^\vee)$. Consequently,
$$
|H_0(K,\wedge^2V^s)| \leq q^{s+\binom s2}=q^{s(s+1)/2}.
$$
This proves \eqref{eq:abelian-schur} when $K\neq1$. When $K=1$, faithfulness and irreducibility force $V=\mathbb{F}_p$, and $L_s=\mathbb{F}_p^s$, so $H_2(L_s,\Z)=\wedge^2\mathbb{F}_p^s$.

In the nonabelian case, $A$ is perfect. The spectral sequence for $1\to A^s\to L_s\to K\to1$ therefore gives
$$
|H_2(L_s,\Z)| \leq |H_2(K,\Z)|\,|H_0(K,H_2(A^s,\Z))|.
$$
Writing $A=S^\rho$ and using $|H_2(S,\Z)|\leq|S|$, we obtain
$$
|H_2(A^s,\Z)|=|H_2(S,\Z)|^{\rho s} \leq |A|^s=N^s.
$$
Lemma \ref{lem:H2-general-bound} completes the proof.
\end{proof}

The preceding automorphism and multiplier bounds give a quadratic saving in the abelian case when $\delta<2$. For general $\delta$, we estimate the event $\mathcal G\twoheadrightarrow L_t$ using the moment of a smaller crown $L_s$.

\begin{lemma}\label{lem:amplification}
Suppose that $L=V\rtimes K$ is of abelian type, and put $D=\End_K(V)\cong\F_q$. For $1\leq s\leq t$,
\begin{equation}\label{eq:sur-count}
|\Sur(L_t,L_s)| \geq \prod_{i=0}^{s-1}(q^t-q^i) \geq c_0q^{st}.
\end{equation}
Consequently, in every decent probability space,
\begin{equation}\label{eq:amplified-bound}
\mu(E(L_t)) \leq c_0^{-1}q^{-st}\int\nu_{L_s}(\mathcal{G})\,d\mu.
\end{equation}
\end{lemma}

\begin{proof}
After identifying $V^t$ with $V\otimes_DD^t$, every surjective $D$-linear map $D^t\twoheadrightarrow D^s$ induces an epimorphism $L_t\twoheadrightarrow L_s$ which fixes the standard complement $K$ pointwise. The number of such maps is the product in \eqref{eq:sur-count}, and
$$
\prod_{i=0}^{s-1}(q^t-q^i)=q^{st}\prod_{i=0}^{s-1}(1-q^{i-t}) \geq c_0q^{st}.
$$
On $E(L_t)$, fix one epimorphism $\mathcal{G}\twoheadrightarrow L_t$ and compose it with these maps. Since the fixed map is surjective, the resulting epimorphisms $\mathcal{G}\twoheadrightarrow L_s$ are distinct. Thus, pointwise,
$$
\mathbf 1_{E(L_t)}(\mathcal{G})\,c_0q^{st} \leq \nu_{L_s}(\mathcal{G}).
$$
Integration gives \eqref{eq:amplified-bound}.
\end{proof}

For an integer $m\geq3$, let $\mathcal L_m$ be as in Corollary \ref{cor:crown-witness-bound}. For $L\in\mathcal L_m$, put
$$
t_m(L)=
\begin{cases}
b(m-2)+1, & L=V\rtimes K\text{ is of abelian type},\\
|\soc(L)|^{m-2}+1, & \soc(L)\text{ is nonabelian},
\end{cases}
$$
where $b=\dim_{\End_K(V)}V$ in the first case. Define
$$
B_m^{\mathrm{ab}}=\bigcup_{\substack{L\in\mathcal L_m\\L\text{ of abelian type}}}E(L_{t_m(L)}),
\qquad
B_m^{\mathrm{nonab}}=\bigcup_{\substack{L\in\mathcal L_m\\L\text{ of nonabelian type}}}E(L_{t_m(L)}).
$$
These are countable unions of measurable events. For fixed $L$, coordinate projections give
$$
E(L_{t+1})\subseteq E(L_t),\qquad
\bigcup_{t>R}E(L_t)=E(L_{R+1})\quad(R\in\Z_{\geq0}).
$$
Thus it suffices to estimate one multiplicity for each $L$.

\begin{proposition}\label{prop:abelian-crown-estimate}
Put $\Delta=\max\{1,\delta\}$ and
\begin{equation}\label{eq:gaussian-coefficient}
\kappa_\delta=
\begin{cases}
1-\delta/2, & 0\leq\delta\leq1,\\
1/(2\delta), & \delta>1.
\end{cases}
\end{equation}
For integers $m\geq3$, set
$$
\beta_m=\kappa_\delta(m-2)^2-(1+\delta)m-\alpha
-\frac{(\alpha+\delta)(m-1)}{\Delta}-\frac\delta2.
$$
For all sufficiently large $m$, depending only on $\alpha,\delta$,
\begin{equation}\label{eq:abelian-final}
\mu(B_m^{\mathrm{ab}})\leq\frac{4M}{c_0}\,2^{-\beta_m}.
\end{equation}
\end{proposition}

\begin{proof}
Write $L=V\rtimes K$, $D=\End_K(V)\cong\F_q$, and $b=\dim_DV$. Then $|V|=q^b$, $K\leq\GL_b(q)$, and $|L_s|=|K|q^{bs}$. For $1\leq s\leq t$, Lemma \ref{lem:amplification}, the moment hypothesis, and Proposition \ref{prop:crown-schur} give
\begin{equation}\label{eq:one-abelian-crown}
\mu(E(L_t))\leq Mc_0^{-1}|K|^{\alpha+\delta m}
q^{-st+\delta s^2/2+c_bs},\qquad
c_b=\alpha b+\frac{\delta(b+1)}2.
\end{equation}
Here we used $d(K)\leq d(L)\leq m$.

Put $a_m=(1+\delta)m+\alpha$. For fixed $q,b$, every abelian datum is represented by a subgroup of $\GL_b(q)$ after choosing an isomorphism $D\cong\F_q$ and a $D$-basis of $V$. We may sum over all subgroups, without imposing irreducibility or the condition $\End_K(V)=\F_q$. Counting generating $m$-tuples gives
\begin{equation}\label{eq:head-count}
\sum_{\substack{K\leq\GL_b(q)\\d(K)\leq m}}|K|^{\alpha+\delta m}
\leq |\GL_b(q)|^{a_m}\leq q^{a_mb^2}.
\end{equation}

Take $t=b(m-2)+1$ and $s=\lfloor t/\Delta\rfloor$. For sufficiently large $m$, uniformly in $b\geq1$, one has $1\leq s\leq t$. If $\delta\leq1$, then $s=t$. If $\delta>1$, then
$$
-st+\frac\delta2s^2=-\frac{t^2}{2\delta}
+\frac\delta2\left(\frac t\delta-s\right)^2.
$$
Since $c_b\leq(\alpha+\delta)b$, in both cases
$$
-st+\frac\delta2s^2+c_bs
\leq-\kappa_\delta t^2+\frac{\alpha+\delta}{\Delta}bt+\frac\delta2.
$$
Using $b(m-2)\leq t\leq b(m-1)$ and $b\geq1$, we obtain
$$
a_mb^2-st+\frac\delta2s^2+c_bs\leq-\beta_mb^2.
$$
It follows from \eqref{eq:one-abelian-crown} and \eqref{eq:head-count} that
\begin{equation}\label{eq:abelian-sum-2}
\mu(B_m^{\mathrm{ab}})\leq Mc_0^{-1}\sum_{b\geq1}\sum_q q^{-\beta_mb^2},
\end{equation}
where $q$ ranges over prime powers. For every real $x\geq2$, the integral test gives
$$
\sum_{n=2}^{\infty}n^{-x}
\leq2^{-x}\left(1+\frac2{x-1}\right)\leq3\cdot2^{-x}.
$$
Hence, when $\beta_m\geq2$,
$$
\sum_{b\geq1}\sum_q q^{-\beta_mb^2}
\leq3\sum_{b\geq1}2^{-\beta_mb^2}
\leq\frac{3\cdot2^{-\beta_m}}{1-2^{-\beta_m}}
\leq4\cdot2^{-\beta_m}.
$$
Since $\beta_m\to\infty$, this proves the proposition.
\end{proof}

\begin{proposition}\label{prop:nonabelian-crown-estimate}
For every integer $m\geq\alpha+\delta+4$,
\begin{equation}\label{eq:nonabelian-final}
\mu(B_m^{\mathrm{nonab}})
\leq M\sum_{N\geq60}N^{-N^{m-2}}
\leq\frac{60M}{59}\,60^{-60^{m-2}}.
\end{equation}
\end{proposition}

\begin{proof}
Write $A=\soc(L)$, $N=|A|\geq60$, and $K=L/A$. For $t=N^{m-2}+1$, Remark \ref{rem:quotient-event} and Propositions \ref{prop:crown-automorphisms} and \ref{prop:crown-schur} give
\begin{equation}\label{eq:one-nonabelian-crown}
\mu(E(L_t))\leq M|K|^{\alpha+\delta m}
\frac{N^{(\alpha+\delta-1)t}}{t!}.
\end{equation}
Since $N>e$ and $t>N^{m-2}$,
$$
t!\geq(t/e)^t\geq N^{(m-3)t},\qquad
\frac{N^{(\alpha+\delta-1)t}}{t!}\leq N^{-2t}.
$$
The last inequality uses $m\geq\alpha+\delta+4$.

Conjugation on $A$ embeds $L$ in $\Aut(A)$, and hence in $\Sym_N$. Since $d(L)\leq m$, the number of possible isomorphism classes of $L$ with this value of $N$ is at most $(N!)^m$, by counting generating $m$-tuples in $\Sym_N$. Also $|K|\leq|L|\leq N!$. Put $a_m=(1+\delta)m+\alpha$. Then
\begin{equation}\label{eq:nonabelian-head-count}
\sum_{\substack{L\in\mathcal L_m\text{ of nonabelian type}\\|\soc(L)|=N}}
|K|^{\alpha+\delta m}\leq(N!)^{a_m}\leq N^{a_mN}.
\end{equation}
Consequently,
$$
\mu(B_m^{\mathrm{nonab}})\leq M\sum_{N\geq60}N^{a_mN-2N^{m-2}}.
$$
Our hypothesis on $m$ gives
$$
a_m\leq(1+m-4)m+(m-4)<m^2\leq60^{m-3}.
$$
Thus $a_mN\leq N^{m-2}$ for every $N\geq60$, proving the first inequality in \eqref{eq:nonabelian-final}. Finally, for integer $N\geq60$ and $m\geq3$,
$$
N^{m-2}\geq60^{m-2}+N-60.
$$
Therefore
\[
\sum_{N\geq60}N^{-N^{m-2}}
\leq60^{-60^{m-2}}\sum_{j\geq0}60^{-j}
=\frac{60}{59}\,60^{-60^{m-2}}.\qedhere
\]
\end{proof}

\begin{proof}[Proof of Theorem \ref{thm:finite-generation}]
By Proposition \ref{prop:finite-quotient-detects-generation} and Theorem \ref{thm:crown-witnesses}, every $\mathcal G\in E_m$ has a quotient $L_{f_L(m)}$ with $L\in\mathcal L_m$. The bounds in Theorem \ref{thm:crown-witnesses} give $f_L(m)\geq t_m(L)$, so coordinate projection yields
\begin{equation}\label{eq:crown-cover}
E_m\subseteq B_m^{\mathrm{ab}}\cup B_m^{\mathrm{nonab}}.
\end{equation}
Propositions \ref{prop:abelian-crown-estimate} and \ref{prop:nonabelian-crown-estimate} imply $\mu(E_m)\to0$. The events $E_m$ are measurable by Proposition \ref{prop:Em-measurable} and decrease to $\{\mathcal G:d(\mathcal G)=\infty\}$. Continuity from above proves the theorem.
\end{proof}

\begin{corollary}\label{cor:gaussian-rank-tail}
Let $(\Omega,\mathcal B,\mu)$ be a decent space with constants $M,\alpha,\delta$, and let $\kappa_\delta$ be as in \eqref{eq:gaussian-coefficient}. There exist constants $C_1,C_2>0$, depending only on $\alpha,\delta$, such that, for every integer $m\geq0$,
$$
\mu\{\mathcal G:d(\mathcal G)>m\}
\leq C_1M\,2^{-\kappa_\delta m^2+C_2m}.
$$
Moreover, for every $0<\lambda<\kappa_\delta$,
$$
\int_\Omega 2^{\lambda d(\mathcal G)^2}\,d\mu(\mathcal G)<\infty.
$$
\end{corollary}

\begin{proof}
In Proposition \ref{prop:abelian-crown-estimate}, one has $\beta_m=\kappa_\delta m^2+O_{\alpha,\delta}(m)$. Combining this with \eqref{eq:crown-cover} and Proposition \ref{prop:nonabelian-crown-estimate} gives the first assertion for all sufficiently large $m$. Increasing $C_1$ covers the remaining integers, since $M\geq1$ and $\mu(E_m)\leq1$.

Theorem \ref{thm:finite-generation} and $\mu\{d(\mathcal G)=r\}\leq\mu(E_{r-1})$ for $r\geq1$ give
$$
\int_\Omega 2^{\lambda d(\mathcal G)^2}\,d\mu
\leq1+C_1M\sum_{r\geq1}
2^{\lambda r^2-\kappa_\delta(r-1)^2+C_2(r-1)}.
$$
The series converges because $\lambda<\kappa_\delta$.
\end{proof}

\subsection{The Liu--Wood moments}
Let $S$ be a finite set of finite groups, and let $\barS$ be the smallest class containing $S$ and closed under subgroups, quotients, and finite direct products. For a profinite group $\mathcal{G}$, put
$$
\mathcal{G}^{\barS}=\varprojlim_{\substack{N\trianglelefteq_{\mathrm{open}}\mathcal{G}\\ \mathcal{G}/N\in\barS}}\mathcal{G}/N.
$$
The Liu--Wood space $\calP$ consists of the isomorphism classes of profinite groups $\mathcal{G}$ for which $\mathcal{G}^{\barS}$ is finite for every finite $S$. Its topology has basic clopen sets
$$
U_{S,J}=\{\mathcal{G}\in\calP: \mathcal{G}^{\barS}\cong J\},
$$
where $J$ ranges over the finite groups in $\barS$, up to isomorphism.

\begin{proposition}\label{prop:LW-moments}
For every finite group $H$ and every $u\in\Z$, the function $\nu_H$ is measurable on $\calP$, and
$$
\int_{\calP}\nu_H(\mathcal{G})\,d\mu_u(\mathcal{G})=|H|^{-u}.
$$
\end{proposition}

\begin{proof}
Take $S=\{H\}$. Every homomorphism $\mathcal{G}\to H$ factors uniquely through $\mathcal{G}^{\barS}$, so
$$
\Sur(\mathcal{G},H)=\Sur(\mathcal{G}^{\barS},H).
$$
It follows that $\mathcal{G}\mapsto|\Sur(\mathcal{G},H)|$ is constant on each basic clopen set $U_{S,J}$, and hence is continuous and measurable.

The finite model satisfies
\begin{equation}\label{eq:LW-finite-moment}
\E\,|\Sur(X_{u,n},H)|=\frac{|\Sur(\wideF_n,H)|}{|H|^{n+u}},
\end{equation}
since an epimorphism $\wideF_n\twoheadrightarrow H$ factors through $X_{u,n}$ precisely when the $n+u$ independent Haar-random relators land in its kernel, an event of probability $|H|^{-(n+u)}$. Since a fixed finite group is generated by a random $n$-tuple with probability tending to $1$,
\begin{equation}\label{eq:finite-moment-limit}
\lim_{n\to\infty}\E|\Sur(X_{u,n},H)|=|H|^{-u}.
\end{equation}
We verify uniform integrability. For a pair of epimorphisms $\varphi,\psi:\wideF_n\twoheadrightarrow H$, the image of $(\varphi,\psi)$ is a subdirect subgroup of $H\times H$. Pairs with image $J$ are in bijection with $\Sur(\wideF_n,J)$. Thus
$$
\E|\Sur(X_{u,n},H)|^2 = \sum_{\substack{J\leq H\times H\\ J\text{ subdirect}}}\frac{|\Sur(\wideF_n,J)|}{|J|^{n+u}} \leq \sum_{\substack{J\leq H\times H\\ J\text{ subdirect}}}|J|^{-u}.
$$
The last sum is finite and independent of $n$. Thus the random variables $|\Sur(X_{u,n},H)|$ are uniformly bounded in $L^2$, hence uniformly integrable.

The laws of $X_{u,n}$ converge weakly to $\mu_u$ \cite[Theorem 1.1]{LW}. Since the epimorphism-counting function is continuous, its values converge in distribution. Uniform integrability and \eqref{eq:finite-moment-limit} now imply the asserted identity.
\end{proof}

\begin{corollary}\label{cor:LW-fg}
For every $u\in\Z$, the Liu--Wood random profinite group is topologically finitely generated with probability $1$.
\end{corollary}

\begin{proof}
Apply Theorem \ref{thm:finite-generation} with $M=1$, $\alpha=\max\{0,-u\}$ and $\delta=0$, using Proposition \ref{prop:LW-moments}.
\end{proof}

\subsection{The Sawin--Wood moments}
Let $M_{g,L}$ be the random closed oriented $3$-manifold obtained from a genus-$g$ Heegaard splitting whose gluing map is a uniform random word of length $L$ in a fixed finite generating set of the mapping class group containing the identity. Sawin and Wood prove that, first letting $L\to\infty$ and then $g\to\infty$, the laws of $\widehat{\pi_1}(M_{g,L})$ converge to a probability measure $\mu_{\mathrm{SW}}$ on their space $\Prof$ of profinite groups having only finitely many open subgroups of each index \cite[Theorem 1.2]{SawinWood3M}. This is the space $\calP$ above, with the same topology \cite[Section 1.3 and Lemma 8.8]{SawinWood3M}. For every finite group $H$, their moment calculation gives
\begin{equation}\label{eq:SW-moment}
\int_{\calP}\nu_H(\mathcal{G})\,d\mu_{\mathrm{SW}}(\mathcal{G})=\frac{|H|\,|H_2(H,\Z)|}{|H_1(H,\Z)|}.
\end{equation}
The limiting moment is computed by Dunfield and Thurston \cite[Theorem 6.21]{DunfieldThurston}; Sawin and Wood's moment-convergence argument \cite[Proposition 4.6 and Section 9.3]{SawinWood3M} shows that it is the moment of the limiting measure, rather than merely the limit of the finite-stage moments.

Since $|H_1(H,\Z)|\geq1$, equation \eqref{eq:SW-moment} implies
$$
\int_{\calP}\nu_H(\mathcal{G})\,d\mu_{\mathrm{SW}}(\mathcal{G}) \leq |H|\cdot|H_2(H,\Z)|.
$$
Thus $(\calP,\mathcal{B},\mu_{\mathrm{SW}})$ is decent with constants $M=1$ and $\alpha=\delta=1$.

\begin{corollary}\label{cor:SW-fg}
A $\mu_{\mathrm{SW}}$-random profinite group is topologically finitely generated with probability $1$.
\end{corollary}

\begin{proof}
Apply Theorem \ref{thm:finite-generation} to the decent space $(\calP,\mathcal{B},\mu_{\mathrm{SW}})$.
\end{proof}

This answers the question posed by Sawin and Wood \cite[Section 10.3]{SawinWood3M}.

\section{Finite presentation}\label{sec:finite-presentation}
We first show that a uniform bound on the number of normal generators in finite quotients gives the same bound in the profinite group.

In the presentation displays in the introduction, $\overline{\langle x_1,\ldots,x_r\rangle}$ denotes the closed normal closure of the relators in the ambient free profinite group; the same convention applies to the defining display for $X_{u,n}$.

\begin{lemma}\label{lem:compactness-relations}
Let $F$ be a profinite group, let $R\trianglelefteq F$ be closed, and let $r\geq0$ be an integer. Suppose that $RU/U$ is normally generated by at most $r$ elements in $F/U$ for every open normal subgroup $U\trianglelefteq F$. Then $R$ is the closed normal closure in $F$ of at most $r$ elements.
\end{lemma}

\begin{proof}
For every open normal subgroup $U\trianglelefteq F$, let
$$
C_U=\{(x_1,\ldots,x_r)\in R^r: x_1U,\ldots,x_rU \text{ normally generate } RU/U \text{ in } F/U\}.
$$
For $r=0$, we interpret $R^0$ as a one-point space. Each $C_U$ is nonempty: lift normal generators of $RU/U$ to $R$ and add identities if necessary. It is clopen, since membership depends only on the image in $(F/U)^r$.

The family $\{C_U\}$ has the finite-intersection property. Indeed, if $U_0=U_1\cap\cdots\cap U_s$, then
$$
C_{U_0} \subseteq C_{U_1}\cap\cdots\cap C_{U_s}.
$$
Compactness of $R^r$ therefore gives a tuple
$$
(x_1,\ldots,x_r)\in\bigcap_UC_U.
$$
Let
$$
R_0=\overline{\langle x_1,\ldots,x_r\rangle^F}.
$$
Then $R_0U=RU$ for every open normal $U$. Since $R_0$ and $R$ are closed,
$$
R_0=\bigcap_U R_0U=\bigcap_U RU=R.
$$
\end{proof}

\subsection{The Sawin--Wood model}
For integers $d,r\geq0$, let
$$
\operatorname{FP}_{d,r}=\Bigl\{\Gp\in\calP: \Gp\cong\wideF_d/\overline{\langle x_1,\ldots,x_r\rangle^{\wideF_d}} \text{ for some } x_1,\ldots,x_r\in\wideF_d\Bigr\}.
$$
Presentations with fewer than $r$ relators are included by using identity relators.

\begin{proposition}\label{prop:SW-measurability}
For all integers $d,r\geq0$, the set $\operatorname{FP}_{d,r}$ is closed in $\calP$. The set of finitely presented groups is Borel. The set of finitely generated groups admitting a presentation with $d(\Gp)$ generators and $d(\Gp)$ relators is also Borel. It remains Borel if we require that every profinite presentation have at least $d(\Gp)$ relators.
\end{proposition}

\begin{proof}
Consider the quotient map
$$
q_{d,r}:\wideF_d^r\longrightarrow\calP, \qquad (x_1,\ldots,x_r)\longmapsto\wideF_d/\overline{\langle x_1,\ldots,x_r\rangle^{\wideF_d}}.
$$
Its image is $\operatorname{FP}_{d,r}$. The map is continuous. Indeed, if $S$ is a finite set of finite groups, put $P_S=(\wideF_d)^{\barS}$ and let $\bar x_i$ denote the image of $x_i$ in $P_S$. The universal property of level completion gives
$$
q_{d,r}(x_1,\ldots,x_r)^{\barS} \cong P_S/\langle \bar x_1,\ldots,\bar x_r\rangle^{P_S}.
$$
Its isomorphism class therefore depends only on the images of the $x_i$ in the finite group $P_S$. Thus the inverse image under $q_{d,r}$ of every basic clopen subset of $\calP$ is clopen.

The space $\calP$ is Hausdorff \cite[Section 1]{SawinWood3M}. Thus $\operatorname{FP}_{d,r}$ is closed, being the image of a compact space under a continuous map. Taking a countable union over $d,r$ proves the first Borel assertion. The sets $\{\Gp:d(\Gp)=d\}$ are Borel by Section \ref{sec:witnesses}. The set of finitely generated groups admitting the stated presentation is
$$
\bigcup_{d\geq0}\bigl(\{\Gp: d(\Gp)=d\}\cap\operatorname{FP}_{d,d}\bigr),
$$
and is Borel as well.
The additional requirement excludes precisely the Borel set
$$
\bigcup_{n,r\geq0}\bigl(\{\Gp:d(\Gp)>r\}\cap\operatorname{FP}_{n,r}\bigr),
$$
consisting of groups admitting a presentation with fewer than $d(\Gp)$ relators. This proves the final assertion. The condition that every finite profinite presentation have at least as many relators as generators is Borel as well: it excludes precisely
$$
\bigcup_{n>r\geq0}\operatorname{FP}_{n,r}.
$$
\end{proof}

\begin{lemma}\label{lem:gaschutz-normal-generation}
Let $P$ be a finite group, let $A\trianglelefteq P$, and let $r\geq1$. For every prime $p$, put
$$
A_p=A/[A,A]A^p,
$$
viewed as an $\F_p[P/A]$-module. If every $A_p$ is generated by at most $r$ elements as a module, then $A$ is normally generated in $P$ by at most $r$ elements.
\end{lemma}

\begin{proof}
The assertion is immediate if $A=1$. Let $\Phi$ be the intersection of the maximal proper subgroups of $A$ that are normal in $P$. A tuple normally generates $A$ in $P$ if and only if its image normally generates $A/\Phi$.

Choose $M_1,\ldots,M_k$ with intersection $\Phi$ such that none can be omitted. For each $i$, the image of $\bigcap_{j<i}M_j$ in $A/M_i$ is nontrivial and normal under $P$. It is therefore all of $A/M_i$. Here the empty intersection means $A$. Induction gives a $P$-equivariant isomorphism
$$
A/\Phi\cong\prod_{i=1}^k A/M_i.
$$
Each factor is either elementary abelian or a direct product of nonabelian simple groups. Write this product as $B\times C$, where $B$ is the product of the abelian factors and $C$ is the product of the nonabelian factors. For every prime $p$, the Sylow $p$-subgroup $B_p$ of $B$ is an $\F_p[P/A]$-module quotient of $A_p$, so it is generated by at most $r$ elements. Combining these tuples over all $p$ gives $r$ elements that normally generate $B$ under $P$.

Extend this tuple to $B\times C$ so that the first element has a nonidentity component in every simple direct factor of $C$. Commutators with each factor show that the normal closure of this element meets that factor nontrivially. It therefore contains the factor. Thus the tuple normally generates $B\times C$ under $P$. Lifting it to $A$ proves the assertion.
\end{proof}

The next proposition follows from Lubotzky's presentation formula \cite[Theorems 3.1 and 5.1]{Lubotzky}. We include a proof. Throughout this section, cohomology and homomorphisms of profinite modules are continuous.

\begin{proposition}\label{prop:presentation-criterion}
Let $\Gp$ be a topologically finitely generated profinite group, put $d=d(\Gp)$, and suppose that for every prime $p$ and every finite simple $\F_p[[\Gp]]$-module $V$, with finite field $D=\End_{\Gp}(V)$, one has
\begin{equation}\label{eq:H2-H1-hypothesis}
\dim_DH^2(\Gp,V) \leq \dim_DH^1(\Gp,V).
\end{equation}
Then $\Gp$ has a profinite presentation with $d$ generators and at most $d$ relations.
\end{proposition}

\begin{proof}
The assertion is immediate if $d=0$, so assume $d\geq1$. Choose an epimorphism
$$
\pi:F:=\wideF_d\twoheadrightarrow\Gp, \qquad R=\ker\pi.
$$
Fix a prime $p$ and a finite simple $\F_p[[\Gp]]$-module $V$. Put $D=\End_{\Gp}(V)$ and
$$
R_p=R/\overline{[R,R]R^p}.
$$
The five-term sequence for $1\to R\to F\to\Gp\to1$, together with $H^2(F,V)=0$, gives
\begin{equation}\label{eq:five-term}
0\longrightarrow H^1(\Gp,V)\longrightarrow H^1(F,V)\longrightarrow\Hom_{\Gp}(R_p,V)\longrightarrow H^2(\Gp,V)\longrightarrow0.
\end{equation}
If $\xi_V=0$ for the trivial module and $\xi_V=1$ otherwise, then
$$
\dim_DH^1(F,V)=(d-\xi_V)\dim_DV.
$$
It follows from \eqref{eq:H2-H1-hypothesis} and \eqref{eq:five-term} that
\begin{equation}\label{eq:hom-Rp-bound}
\dim_D\Hom_{\Gp}(R_p,V) \leq (d-\xi_V)\dim_D(V) \leq d\cdot\dim_D(V).
\end{equation}
Let $U\trianglelefteq F$ be open and normal, and put
$$
P_U=F/U,\qquad A_U=RU/U,\qquad Q_U=P_U/A_U=F/RU.
$$
For every prime $p$, the $\F_p[Q_U]$-module
$$
A_{U,p}=A_U/[A_U,A_U]A_U^p
$$
is a quotient of $R_p$. If $V$ is a simple $\F_p[Q_U]$-module, inflate it to a $\Gp$-module. Composition with $R_p\twoheadrightarrow A_{U,p}$ gives
$$
\dim_D\Hom_{Q_U}(A_{U,p},V) \leq \dim_D\Hom_{\Gp}(R_p,V) \leq d\cdot\dim_D(V).
$$
For a finite group $Q$, an $\F_p[Q]$-module $N$ of finite dimension is generated by at most $d$ elements if and only if
$$
\dim_{\End_Q(V)}\Hom_Q(N,V) \leq d\dim_{\End_Q(V)}V
$$
for every simple module $V$. Indeed, let $J$ be the Jacobson radical of $\F_p[Q]$. The multiplicity of $V$ in $N/JN$ is $\dim_{\End_Q(V)}\Hom_Q(N,V)$, while its multiplicity in $\F_p[Q]/J$ is $\dim_{\End_Q(V)}V$. The assertion follows from Nakayama's lemma. Thus every $A_{U,p}$ is generated by at most $d$ elements. Lemma \ref{lem:gaschutz-normal-generation} shows that $A_U$ is normally generated by at most $d$ elements in $P_U$.

This holds for every open normal $U\trianglelefteq F$. Lemma \ref{lem:compactness-relations} gives $d$ elements whose closed normal closure in $F$ is $R$, proving the assertion.
\end{proof}

\begin{proof}[Proof of Theorem \ref{thm:SW-finite-presentation}]
Sawin and Wood prove that the support of $\mu_{\mathrm{SW}}$ is the closure in $\calP$ of the profinite completions of fundamental groups of closed oriented $3$-manifolds \cite[Theorems 1.2 and 1.5]{SawinWood3M}. For every $\Gp$ in this support, there is a homomorphism
$$
\tau:H^3(\Gp,\Q/\Z)\longrightarrow\Q/\Z
$$
with the following properties. For every prime $p$ and every finite simple $\F_p[[\Gp]]$-module $V$, put $V^\vee=\Hom_{\F_p}(V,\F_p)$. Then
\begin{align}
&\dim_{\F_p} H^1(\Gp,V)=\dim_{\F_p} H^1(\Gp,V^\vee),\label{eq:PD-dims}\\
&H^2(\Gp,V)\times H^1(\Gp,V^\vee)\longrightarrow\Q/\Z,\qquad(\alpha,\beta)\longmapsto\tau(\alpha\smile\beta),\label{eq:PD-pairing}
\end{align}
where the pairing is nondegenerate in its first variable. Here the cup product uses evaluation $V\otimes_{\F_p}V^\vee\to\F_p$ followed by the homomorphism $\F_p\to\Q/\Z$ sending $1$ to $1/p$.

Suppose now that $d(\Gp)<\infty$. The group $H^1(\Gp,V^\vee)$ is finite, since a continuous $1$-cocycle is determined by its values on a finite generating tuple. The pairing \eqref{eq:PD-pairing} gives an injection
$$
H^2(\Gp,V)\hookrightarrow\Hom\bigl(H^1(\Gp,V^\vee),\Q/\Z\bigr).
$$
Thus $H^2(\Gp,V)$ is finite and, by \eqref{eq:PD-dims},
$$
\dim_{\F_p}H^2(\Gp,V) \leq \dim_{\F_p}H^1(\Gp,V).
$$
Dividing by $[D:\F_p]$, where $D=\End_{\Gp}(V)$, gives \eqref{eq:H2-H1-hypothesis}. Proposition \ref{prop:presentation-criterion} therefore gives a presentation with $d(\Gp)$ generators and at most $d(\Gp)$ relators.

By Corollary \ref{cor:SW-fg}, the hypothesis $d(\Gp)<\infty$ holds for $\mu_{\mathrm{SW}}$-almost every $\Gp$. The support has full measure: $\calP$ has a countable clopen basis, so the complement of the support is a countable union of null basic opens. Thus this presentation exists almost surely.

To prove that the number of relators is minimal, fix a prime $p$. For $k\geq1$, let $C_{p^k}$ be the cyclic group of order $p^k$. Its Schur multiplier is trivial, so \eqref{eq:SW-moment} gives
$$
\E_{\mu_{\mathrm{SW}}}\nu_{\Z/p^k\Z}=1.
$$
The free action of $\Aut(\Z/p^k\Z)$ on epimorphisms, as in Remark \ref{rem:quotient-event}, therefore gives
$$
\mu_{\mathrm{SW}}\{\Gp:\Gp\twoheadrightarrow \Z/p^k\Z\}
\leq\frac{1}{p^{k-1}(p-1)}.
$$
These events are measurable and decreasing in $k$, so their intersection has measure zero. The maximal abelian pro-$p$ quotient $A_p$ of a finitely generated profinite group is a finitely generated $\Z_p$-module. If $A_p$ were infinite, it would have a quotient isomorphic to $\Z_p$, and hence quotients $C_{p^k}$ for every $k$. Thus $A_p$ is finite almost surely.

Let $\Gp$ be a finitely generated group with this property. A presentation with infinitely many generators and finitely many relators has an infinite elementary abelian $p$-quotient. It therefore cannot present $\Gp$. Consider a presentation of $\Gp$ with $n$ generators and $r$ relators $y_1,\ldots,y_r\in\wideF_n$. Passing to the maximal abelian pro-$p$ quotient gives
$$
A_p\cong\Z_p^n/\sum_{i=1}^r\Z_p\bar y_i,
$$
where $\bar y_i$ is the image of $y_i$ in $\Z_p^n$. Since $A_p$ is finite, these vectors span $\Q_p^n$. Hence $r\geq n\geq d(\Gp)$. Every profinite presentation of $\Gp$ therefore has at least $d(\Gp)$ relators. The presentation already constructed attains this bound with $d(\Gp)$ generators. Proposition \ref{prop:SW-measurability} gives the asserted measurability.
\end{proof}

\subsection{The Liu--Wood model}
We now specialize to the probability space $\calP$ of small profinite groups. Retain the notation $\barS$ from the end of Section \ref{sec:estimates}, and write
$$
F_{d,S}:=(\wideF_d)^{\barS}.
$$
For finite $S$, this is the finite free pro-$\barS$ group on $d$ generators.

We give an alternative, cohomological proof of Liu and Wood's \cite[Lemma 13.1]{LW}.

\begin{lemma}\label{lem:finite-level-relations}
Let $S$ be a finite set of finite groups and let $H$ be a finite pro-$\barS$ group. Let $n\geq d\geq d(H)$ and $r\geq1$ be integers. Suppose that some epimorphism $F_{n,S}\twoheadrightarrow H$ has kernel normally generated by at most $n-d+r$ elements. Then the kernel of every epimorphism $F_{d,S}\twoheadrightarrow H$ is normally generated by at most $r$ elements.
\end{lemma}

\begin{proof}
Adding one generator and one relator if necessary, we may assume $n>d$. Put $P_j=F_{j,S}$ for $j=d,n$. Choose epimorphisms $\alpha_j:P_j\twoheadrightarrow H$, with $\alpha_n$ as in the hypothesis, and put $A_j=\ker\alpha_j$.

Fix a prime $p$ and a finite simple $\F_p[H]$-module $V$. Write $D=\End_H(V)$, $v=\dim_DV$, and
$$
m_j(V)=\dim_D\Hom_H\bigl(A_j/[A_j,A_j]A_j^p,V\bigr).
$$
The hypothesis gives $m_n(V)\leq(n-d+r)v$. We shall show that $m_d(V)\leq rv$.

We may assume $m_d(V)>0$. A nonzero $H$-equivariant map from $A_d/[A_d,A_d]A_d^p$ to $V$ is surjective. Pulling its kernel back to $A_d$ gives a subgroup normal in $P_d$, and hence an extension
$$
1\longrightarrow V\longrightarrow E\longrightarrow H\longrightarrow1
$$
with $E\in\barS$. The diagonal copy $\Delta V$ is normal in $E\times_H E$, and
$$
(E\times_H E)/\Delta V\cong V\rtimes H.
$$
Indeed, the diagonal copy of $E$ gives a section after quotienting by $\Delta V$. Thus $V\rtimes H\in\barS$.

By the universal property of $P_j$, the lifts of its distinguished generators to $V\rtimes H$ can be chosen arbitrarily. Hence a $1$-cocycle $P_j\to V$ is specified by any $j$ values in $V$. If $\xi_V=0$ for the trivial module and $\xi_V=1$ otherwise, then
$$
\dim_DH^1(P_j,V)=(j-\xi_V)v.
$$
Put
$$
K_j=\ker\bigl(H^2(H,V)\longrightarrow H^2(P_j,V)\bigr).
$$
The five-term sequence gives
$$
m_j(V)=(j-\xi_V)v-\dim_DH^1(H,V)+\dim_DK_j.
$$

We claim that $K_d=K_n$. A nonzero class $c\in H^2(H,V)$ corresponds to a nonsplit extension $E_c$ of $H$ by $V$. Every subgroup of $E_c$ mapping onto $H$ is all of $E_c$. To see this, its intersection with $V$ is an $H$-submodule, hence either $V$ or $0$. The first case gives the whole group, and the second gives a splitting. Now $c\in K_j$ if and only if $\alpha_j$ lifts to $E_c$. Any such lift is surjective, so it implies $E_c\in\barS$. Conversely, if $E_c\in\barS$, the universal property of $P_j$ gives a lift. Thus membership in $K_j$ is independent of $j$, proving the claim.

It follows that
$$
m_d(V)=m_n(V)-(n-d)v\leq rv.
$$
This also holds when $m_d(V)=0$. The module criterion in the proof of Proposition \ref{prop:presentation-criterion} shows that every $A_d/[A_d,A_d]A_d^p$ is generated by at most $r$ elements. Lemma \ref{lem:gaschutz-normal-generation} now gives the assertion.
\end{proof}

Put
$$
\operatorname{FP}:=\bigcup_{d,r\geq0}\operatorname{FP}_{d,r}.
$$

\begin{proposition}\label{prop:LW-measurability}
The Borel $\sigma$-algebra of $\calP$ is generated by the functions $\nu_A$, where $A$ ranges over the finite groups. In particular, $\operatorname{FP}_{d,r}$ and $\operatorname{FP}$ belong to this $\sigma$-algebra.
\end{proposition}

\begin{proof}
Each $\nu_A$ is continuous: its value is determined by the finite completion $\Gp^{\overline{\{A\}}}$. Conversely, for a finite set $S$ and a finite pro-$\barS$ group $H$,
\begin{equation}\label{eq:clopen-via-moments}
U_{S,H}=\bigcap_{A\in\barS}\{\Gp\in\calP:\nu_A(\Gp)=\nu_A(H)\}.
\end{equation}
Indeed, all maps to groups in $\barS$ factor through $\Gp^{\barS}$. If the equalities on the right hold, put $K=\Gp^{\barS}$. Both $H$ and $K$ are finite groups in $\barS$; taking $A=H$ and $A=K$ gives epimorphisms in both directions, so $H\cong K$. This proves \eqref{eq:clopen-via-moments}. The intersections are countable, and the sets $U_{S,H}$ form a countable basis. The two $\sigma$-algebras therefore agree. The final assertion follows from Proposition \ref{prop:SW-measurability}.
\end{proof}

For a nonabelian finite simple group $S$ and $\Gp\in\calP$, put
\begin{equation}\label{eq:def-QS}
Q_S(\Gp):=\frac{\nu_S(\Gp)}{|\Aut(S)|}.
\end{equation}
This counts the open normal subgroups with quotient isomorphic to $S$.

\begin{lemma}\label{lem:simple-poisson}
Let $S_1,\ldots,S_e$ be pairwise nonisomorphic nonabelian finite simple groups and let $t_1,\ldots,t_e\geq0$. For every $\Gp\in\calP$,
\begin{equation}\label{eq:goursat-count}
\nu_{S_1^{t_1}\times\cdots\times S_e^{t_e}}(\Gp)=\prod_{i=1}^e|\Aut(S_i)|^{t_i}\bigl(Q_{S_i}(\Gp)\bigr)_{t_i},
\end{equation}
where $(x)_t=x(x-1)\cdots(x-t+1)$. Under $\mu_u$, the variables $(Q_S)_S$ are independent, with
\begin{equation}\label{eq:simple-Poisson-PFG}
Q_S\sim\operatorname{Pois}(\lambda_{S,u}),\qquad
\lambda_{S,u}=\frac{|S|^{-u}}{|\Aut(S)|}=\frac{|S|^{-u-1}}{|\Out(S)|}.
\end{equation}
\end{lemma}

\begin{proof}
By Goursat's lemma, a tuple of epimorphisms to nonabelian simple groups maps onto their product if and only if its coordinate kernels are pairwise distinct. There are $|\Aut(S_i)|$ epimorphisms with each prescribed kernel and quotient $S_i$. Counting ordered choices of distinct kernels proves \eqref{eq:goursat-count}. Proposition \ref{prop:LW-moments} gives
$$
\E_{\mu_u}\prod_{i=1}^e(Q_{S_i})_{t_i}=\prod_{i=1}^e\lambda_{S_i,u}^{t_i}.
$$
These mixed factorial moments determine the joint law. To see this directly, for integers $k_i\geq0$ use the pointwise identity
$$
\prod_{i=1}^e\mathbf 1_{\{Q_{S_i}=k_i\}}
=\sum_{t_1\geq k_1,\ldots,t_e\geq k_e}
\prod_{i=1}^e(-1)^{t_i-k_i}\binom{t_i}{k_i}\binom{Q_{S_i}}{t_i}.
$$
The expectations of the absolute values sum to
$$
\prod_{i=1}^e\sum_{t_i\geq k_i}\binom{t_i}{k_i}\frac{\lambda_{S_i,u}^{t_i}}{t_i!}<\infty.
$$
Taking expectations termwise therefore gives
$$
\mu_u\{Q_{S_i}=k_i\ (1\leq i\leq e)\}
=\prod_{i=1}^e e^{-\lambda_{S_i,u}}\frac{\lambda_{S_i,u}^{k_i}}{k_i!}.
$$
This proves the assertion for every finite subfamily, and hence for the whole family.
\end{proof}

\begin{lemma}\label{lem:positive-relator-count}
If $u<0$, then $d(\Gp)>-u$ for $\mu_u$-almost every $\Gp$.
\end{lemma}

\begin{proof}
Put $k=-u$. For $n\geq7$, let $S_n=A_n$ and
$$
\lambda_n=\frac{|S_n|^{k-1}}2,
$$
using $|\Out(A_n)|=2$. If $d(\Gp)\leq k$, precomposition with an epimorphism $\wideF_k\twoheadrightarrow\Gp$ gives
$$
Q_{S_n}(\Gp)\leq\frac{|\Sur(\wideF_k,S_n)|}{|\Aut(S_n)|}.
$$
For $k=1$ the right-hand side is zero. By Lemma \ref{lem:simple-poisson}, the variables $Q_{S_n}$ are independent Poisson variables of mean $1/2$. The probability that the first $N$ of them vanish is $e^{-N/2}$, so $\mu_{-1}\{d(\Gp)\leq1\}=0$.

Suppose $k\geq2$. Every $k$-tuple in a fixed point stabilizer of $A_n$ fails to generate $A_n$, so
$$
\frac{|\Sur(\wideF_k,S_n)|}{|\Aut(S_n)|}
\leq\lambda_n(1-n^{-k}).
$$
Since $Q_{S_n}$ has mean and variance $\lambda_n$, Chebyshev's inequality gives
$$
\mu_{-k}\{d(\Gp)\leq k\}
\leq\mu_{-k}\{Q_{S_n}\leq\lambda_n(1-n^{-k})\}
\leq\frac{n^{2k}}{\lambda_n}\longrightarrow0.
$$
The last limit follows from $|S_n|=n!/2$ and $k\geq2$.
\end{proof}

\begin{proof}[Proof of Theorem \ref{thm:finite-presentation}]
Fix a nontrivial finite group $B$. The union of the sets $U_{\{B\},H}$ of positive $\mu_u$-measure has measure one. Its complement is a countable union of null sets. Intersect these unions over the isomorphism classes of nontrivial finite groups $B$. By Corollary \ref{cor:LW-fg} and Lemma \ref{lem:positive-relator-count}, we obtain a measurable set $\Omega_0$ of measure one such that every $\Gp\in\Omega_0$ satisfies
$$
d(\Gp)<\infty,\qquad d(\Gp)>-u\ \text{if }u<0,
$$
and
$$
\mu_u\bigl(U_{\{B\},\Gp^{\overline{\{B\}}}}\bigr)>0
\qquad\text{for every nontrivial finite group }B.
$$

Fix $\Gp\in\Omega_0$ and put $d=d(\Gp)$. If $d=0$, then $\Gp=1$ and $u\geq0$, so the presentation with no generators and $u$ identity relators suffices. Otherwise $r:=d+u\geq1$. Choose an epimorphism
$$
\pi:F:=\wideF_d\twoheadrightarrow\Gp,\qquad R=\ker\pi.
$$
For a proper open normal subgroup $U\trianglelefteq F$, put
$$
B_U=F/U,\qquad S_U=\{B_U\},\qquad
P_U=F^{\overline{S_U}},\qquad H_U=\Gp^{\overline{S_U}}.
$$
The map $\pi$ induces an epimorphism $\psi_U:P_U\twoheadrightarrow H_U$. Since $\mu_u(U_{S_U,H_U})>0$, convergence of the finite Liu--Wood models gives an epimorphism
$$
F_{n,S_U}\twoheadrightarrow H_U
$$
whose kernel is normally generated by at most $n+u$ elements, for some $n\geq d$. Since $n+u=n-d+r$ and $d(H_U)\leq d$, Lemma \ref{lem:finite-level-relations} shows that $\ker\psi_U$ is normally generated by at most $r$ elements.

Let $p_U:P_U\twoheadrightarrow B_U$ be induced by $F\twoheadrightarrow B_U$. We claim that
\begin{equation}\label{eq:kernel-image}
p_U(\ker\psi_U)=RU/U.
\end{equation}
The image of $R$ in $P_U$ lies in $\ker\psi_U$, proving one inclusion. For the reverse inclusion, let $f\in F$ represent an element of $\ker\psi_U$. Then $\pi(f)$ maps trivially to $H_U$. The group $F/RU$ is a quotient of both $\Gp$ and $B_U$, so it belongs to $\overline{S_U}$ and the map $\Gp\twoheadrightarrow F/RU$ factors through $H_U$. Thus $f\in RU$, proving \eqref{eq:kernel-image}.

It follows that $RU/U$ is normally generated by at most $r$ elements in $F/U$. For $U=F$ this is automatic. Lemma \ref{lem:compactness-relations} therefore shows that $R$ is the closed normal closure of at most $r$ elements. Adding if necessary the identity a few times, gives a presentation of $\Gp$ with $d$ generators and $d+u$ relators.

Finally, put $d_0=0$ if $u\geq0$, and $d_0=1-u$ if $u<0$. The following presentation event is
$$
\bigcup_{d\geq d_0}\bigl(\{\Gp:d(\Gp)=d\}\cap\operatorname{FP}_{d,d+u}\bigr).
$$
It is Borel by Section \ref{sec:witnesses} and Proposition \ref{prop:LW-measurability}, and contains $\Omega_0$.
\end{proof}

\section{Generator-rank asymptotics}\label{sec:generator-rank-asymptotics}
For $u\in\Z$ and $r\geq0$, put
$$
\mathsf p_u(r):=\mu_u\{\Gp\in\calP:d(\Gp)=r\}.
$$
These sets are measurable by Proposition \ref{prop:Em-measurable}. For $r\geq1$ they equal $E_{r-1}\setminus E_r$, and for $r=0$ they equal $\calP\setminus E_0$. We prove Theorem \ref{thm:generator-rank-first-order} by comparing $d(\Gp)$ with the rank of its maximal elementary abelian $2$-quotient.

For a prime power $q$ and $j\geq0$, set
\begin{equation}\label{eq:eta-q}
\eta_q(j):=\prod_{i=1}^j(1-q^{-i}),\qquad \eta_q(\infty):=\prod_{i=1}^{\infty}(1-q^{-i}).
\end{equation}
Thus $c_0=\eta_2(\infty)$. For an integer $a$ and $h\geq0$, put
\begin{equation}\label{eq:def-pi-qa}
\pi_{q,a}(h):=
\begin{cases}
\dfrac{\eta_q(\infty)}{\eta_q(h)\,\eta_q(h+a)}\,q^{-h(h+a)}, & h+a\geq0,\\[2mm]
0, & h+a<0.
\end{cases}
\end{equation}

\subsection{Crown witnesses}
Let $\mathscr L$ contain one representative of each isomorphism class of finite monolithic primitive groups. We extend the definition of $f_L(m)$ from Section \ref{sec:crowns} by putting $f_L(m)=1$ when $d(L)>m$. Thus, for every $m\geq0$,
$$
f_L(m)=\min\{t\geq1:d(L_t)>m\}.
$$
For $L\in\mathscr L$, put
\begin{equation}\label{eq:def-WLm}
W_L(m):=\{\Gp\in\calP:\Gp\twoheadrightarrow L_{f_L(m)}\}.
\end{equation}
The function $f_L(m)$ is nondecreasing in $m$, and coordinate projections give $W_L(m+1)\subseteq W_L(m)$.

\begin{proposition}\label{prop:exact-crown-filtration}
For every integer $m\geq2$,
\begin{equation}\label{eq:exact-filtration}
E_m=\bigcup_{\substack{L\in\mathscr L\\d(L)\leq m}}W_L(m)=\bigcup_{L\in\mathscr L}W_L(m).
\end{equation}
\end{proposition}

\begin{proof}
If $\Gp\twoheadrightarrow L_{f_L(m)}$, then $d(\Gp)>m$ by the definition of $f_L(m)$. Conversely, if $d(\Gp)>m$, Proposition \ref{prop:finite-quotient-detects-generation} gives a finite quotient $H$ with $d(H)>m$. Theorem \ref{thm:crown-witnesses} then gives $L\in\mathscr L$ with $d(L)\leq m$ and $H\twoheadrightarrow L_{f_L(m)}$.
\end{proof}

For $L=\mathbb{F}_2$, one has $f_{\mathbb{F}_2}(m)=m+1$, so
\begin{equation}\label{eq:C2-witness}
W_{\mathbb{F}_2}(m)=\{\Gp:\Gp\twoheadrightarrow \mathbb{F}_2^{m+1}\}.
\end{equation}
We compute the distribution of these quotients and then bound the contribution of the remaining groups $L$.

\subsection{The elementary abelian \texorpdfstring{$2$}{2}-quotient}
For $\Gp\in\calP$, the quotient $\Gp^{\overline{\{\mathbb{F}_2\}}}$ is a finite elementary abelian $2$-group. Put
$$
N(\Gp):=\dim_{\F_2}\Hom(\Gp,\mathbb{F}_2)=\dim_{\F_2}H^1(\Gp,\F_2).
$$
Here homomorphisms are continuous, and $\mathbb{F}_2$ is identified with the additive group of the trivial module $\F_2$. Then
$$
\Gp^{\overline{\{\mathbb{F}_2\}}}\cong \mathbb{F}_2^{N(\Gp)},\qquad
\{N=h\}=U_{\{\mathbb{F}_2\},\mathbb{F}_2^h}.
$$
In particular, $N$ is finite and measurable on $\calP$.

\begin{proposition}\label{prop:F2-rank-law}
For every $u\in\Z$ and $t\geq0$,
\begin{equation}\label{eq:F2-moments}
\E_{\mu_u}\prod_{i=0}^{t-1}\bigl(2^N-2^i\bigr)=2^{-ut}.
\end{equation}
Moreover, for every $h\geq0$,
\begin{equation}\label{eq:F2-law}
\mu_u\{N=h\}=\pi_{2,u}(h).
\end{equation}
\end{proposition}

\begin{proof}
A homomorphism $\Gp\to \mathbb{F}_2^t$ is surjective if and only if its $t$ coordinates are linearly independent in $\Hom(\Gp,\mathbb{F}_2)$. Hence
\begin{equation}\label{eq:C2-count}
\nu_{\mathbb{F}_2^t}(\Gp)=\prod_{i=0}^{t-1}\bigl(2^{N(\Gp)}-2^i\bigr).
\end{equation}
Proposition \ref{prop:LW-moments} gives \eqref{eq:F2-moments}.

For the distribution, consider the finite model $X_{u,n}$, with $n+u\geq0$. The images of its $n+u$ Haar-random relators in $\wideF_n^{\overline{\{\mathbb{F}_2\}}}=\mathbb{F}_2^n$ are independent and uniform. Thus $N(X_{u,n})$ is the dimension of the cokernel of a uniform $n$ by $n+u$ matrix over $\F_2$. The number of such matrices of rank $k$ is
$$
\prod_{i=0}^{k-1}\frac{(2^n-2^i)(2^{n+u}-2^i)}{2^k-2^i}.
$$
Indeed, a rank-$k$ map factors as a surjection $\F_2^{n+u}\twoheadrightarrow\F_2^k$ followed by an injection $\F_2^k\hookrightarrow\F_2^n$. Each map has $|\GL_k(2)|$ such factorizations. Taking $k=n-h$ and dividing by $2^{n(n+u)}$ gives
$$
\Pr\{N(X_{u,n})=h\}
=\frac{\eta_2(n)\eta_2(n+u)}{\eta_2(h)\eta_2(h+u)\eta_2(n-h)}\,2^{-h(h+u)}
$$
for $\max\{0,-u\}\leq h\leq n$, and the probability is zero otherwise. Since $\{N=h\}$ is clopen, weak convergence of $X_{u,n}$ to $\mu_u$ allows us to let $n\to\infty$. The result is \eqref{eq:F2-law}.
\end{proof}

\subsection{The remaining witnesses}
Put $a=\max\{0,-u\}$, so that $|H|^{-u}\leq|H|^a$ for every finite group $H$.

\begin{lemma}\label{lem:non-C2-witnesses}
For integers $m\geq3$, put
$$
\mathscr W_m:=\bigcup_{\substack{L\in\mathscr L,\ L\not\cong \mathbb{F}_2\\d(L)\leq m}}W_L(m).
$$
There are constants $C_u,m_u>0$, depending only on $u$, such that for all integers $m\geq m_u$,
\begin{equation}\label{eq:non-C2-bound}
\mu_u(\mathscr W_m)\leq3^{-(m-2)^2+C_um}.
\end{equation}
\end{lemma}

\begin{proof}
Suppose first that $L=V\rtimes K$ is of abelian type, with $d(L)\leq m$ and $L\not\cong \mathbb{F}_2$. Write $q=|\End_K(V)|$, $b=\dim_{\End_K(V)}V$ and $t=f_L(m)$. By \eqref{eq:abelian-f-gap}, one has $t>b(m-2)$. If $(q,b)=(2,1)$, then $K\leq\GL_1(2)=1$ and $L=\mathbb{F}_2$, contrary to our assumption. Thus $q^{b^2}\geq3$.

Remark \ref{rem:quotient-event} and Propositions \ref{prop:crown-automorphisms} and \ref{prop:LW-moments} give
$$
\mu_u(W_L(m))\leq\frac{\E_{\mu_u}\nu_{L_t}}{|\Aut(L_t)|}
\leq\frac{|L_t|^{-u}}{c_0q^{t^2}}
\leq c_0^{-1}|K|^aq^{abt-t^2}.
$$
For $m\geq a/2+3$, the exponent $abt-t^2$ is decreasing in $t$ on $t\geq b(m-2)$, so
$$
\mu_u(W_L(m))\leq c_0^{-1}|K|^aq^{ab^2(m-2)-b^2(m-2)^2}.
$$
As in \eqref{eq:head-count}, counting generating $m$-tuples bounds the number of possible subgroups $K\leq\GL_b(q)$ by $q^{mb^2}$. Also $|K|^a\leq q^{ab^2}$. Put
$$
A_m=(m-2)^2-a(m-2)-(m+a).
$$
Summing over $L$ with fixed $(q,b)$ gives at most $c_0^{-1}q^{-A_mb^2}$. Each integer $\kappa=q^{b^2}$ arises from at most $1+\log_2\kappa\leq\kappa$ pairs $(q,b)$. Therefore the contribution from abelian $L$ is at most
$$
c_0^{-1}\sum_{\kappa\geq3}\kappa^{-(A_m-1)}
\leq4c_0^{-1}\,3^{-(A_m-1)}
$$
when $A_m\geq5$, by comparison with the integral from $3$ to $\infty$. Since $A_m=(m-2)^2-O_u(m)$, this is at most $\frac12\,3^{-(m-2)^2+C_um}$ for a suitable $C_u$ and all sufficiently large $m$.

Now suppose that $L$ is of nonabelian type and $d(L)\leq m$. Put $v=|\soc(L)|$. By \eqref{eq:nonabelian-f-gap}, one has $f_L(m)>v^{m-2}$. Coordinate projection gives
$$
W_L(m)\subseteq E(L_{v^{m-2}+1}).
$$
Thus the contribution from nonabelian $L$ is at most $\mu_u(B_m^{\mathrm{nonab}})$. Proposition \ref{prop:nonabelian-crown-estimate}, with $M=1$, $\alpha=a$ and $\delta=0$, gives
$$
\mu_u(B_m^{\mathrm{nonab}})\leq\frac{60}{59}\,60^{-60^{m-2}}
$$
for $m\geq a+4$. This is smaller than $\frac12\,3^{-(m-2)^2+C_um}$ for all sufficiently large $m$. Adding the two estimates proves \eqref{eq:non-C2-bound}.
\end{proof}

\begin{proof}[Proof of Theorem \ref{thm:generator-rank-first-order}]
For every $\Gp\in\calP$, one has $N(\Gp)\leq d(\Gp)$, since $\mathbb{F}_2^{N(\Gp)}$ is a quotient of $\Gp$. We claim that, for $r\geq4$,
\begin{equation}\label{eq:symmetric-difference}
\{d(\Gp)=r\}\,\triangle\,\{N(\Gp)=r\}
\subseteq\mathscr W_{r-1}\cup\mathscr W_r.
\end{equation}
If $d(\Gp)=r$ and $N(\Gp)<r$, Proposition \ref{prop:exact-crown-filtration} gives a witness $W_L(r-1)$ with $d(L)\leq r-1$. By \eqref{eq:C2-witness}, $L$ cannot be $\mathbb{F}_2$, so $\Gp\in\mathscr W_{r-1}$. If $N(\Gp)=r$ and $d(\Gp)>r$, the same argument at $m=r$ gives $\Gp\in\mathscr W_r$. This proves \eqref{eq:symmetric-difference}.

Lemma \ref{lem:non-C2-witnesses} now gives, for all sufficiently large $r$,
\begin{align*}
\bigl|\mathsf p_u(r)-\mu_u\{N=r\}\bigr|
&\leq\mu_u(\mathscr W_{r-1})+\mu_u(\mathscr W_r)\\
&\leq2\cdot3^{-(r-3)^2+C_u(r-1)}
\leq3^{-r^2+C_u'r}
\end{align*}
for a suitable $C_u'$. Proposition \ref{prop:F2-rank-law} gives $\mu_u\{N=r\}=\pi_{2,u}(r)$, proving the asserted error bound. Finally, $\eta_2(r)$ and $\eta_2(r+u)$ tend to $\eta_2(\infty)$, while
$$
3^{-r^2+C_u'r}=o\bigl(2^{-r(r+u)}\bigr).
$$
This gives the first assertion of the theorem.
\end{proof}

\begin{remark}\label{rem:higher-order}
We expect further terms in the asymptotic expansion of $\mathsf p_u(r)$ to be indexed by monolithic primitive groups. The crown powers of $C_3$ and $S_3=\mathbb{F}_3\rtimes \mathbb{F}_2$ are the next witnesses to consider. We do not pursue this here.
\end{remark}

\section{Universal perfect profinite groups}\label{sec:perfect-universal}
Nikolov asked whether, for each $d>1$, there is a finitely generated perfect profinite group whose finite quotients include every finite perfect $d$-generated group \cite[Conjecture B1]{NikolovFiniteImages}. We prove the following stronger statement. Here and below, ``$d$-generated'' means ``generated by at most $d$ elements.''

\begin{theorem}\label{thm:nikolov-B1-full}
Let $d\geq2$.
\begin{enumerate}[label=\textup{(\roman*)}]
\item Every finite family $H_1,\ldots,H_s$ of finite $d$-generated perfect groups has a finite $d$-generated perfect common cover: there is a finite perfect group $C$, with $d(C)\leq d$, and epimorphisms
$$
C\twoheadrightarrow H_i\qquad(1\leq i\leq s).
$$
\item There is a topologically $d$-generated profinite group $U_d$ such that
$$
U_d=[U_d,U_d]
$$
as an abstract group and whose finite continuous quotients are precisely the finite perfect groups $H$ with $d(H)\leq d$.
\item $U_d$ is unique: if $\mathcal V$ is a topologically $d$-generated profinite group with $\mathcal V=\overline{[\mathcal V,\mathcal V]}$ whose finite continuous quotients include every finite perfect $d$-generated group, then $\mathcal V\cong U_d$ as topological groups.
\item $U_d$ is projective, $d(U_d)=d$, and $U_d$ admits a profinite presentation with $d$ generators and $d$ relators. Every profinite presentation of $U_d$ has at least $d$ relators.
\end{enumerate}
\end{theorem}

To construct $U_d$, we show that any two finite perfect $d$-generated groups have a finite perfect $d$-generated common cover. We construct the common cover by induction along a chief series of one of the two groups, applying Lemma \ref{lem:amalgamation} at each step. Its proof uses Frattini subgroups in the central case and Gasch\"utz's lifting lemma in the noncentral case.

We first record the Frattini observation used in the central case.

\begin{lemma}\label{lem:central-frattini}
Let $P$ be a finite perfect group. Then every central subgroup $Z\leq Z(P)$ is contained in $\Frat(P)$. Consequently, if $P\twoheadrightarrow Q$ has central kernel, then
$$
d(P)=d(Q).
$$
\end{lemma}

\begin{proof}
Suppose that a maximal subgroup $M<P$ does not contain $Z$. Then $P=MZ$. Since $Z$ is central, $M$ is normal in $P$, so $P/M$ is a nontrivial abelian quotient of $P$, contrary to $P=[P,P]$. Thus every maximal subgroup contains $Z$, and hence $Z\leq\Frat(P)$.

For the last assertion, a subset of $P$ generates $P$ if and only if its image generates $P/Z$ whenever $Z\leq\Frat(P)$. Applying this to the central kernel of $P\twoheadrightarrow Q$ gives $d(P)=d(Q)$.
\end{proof}

\begin{lemma}\label{lem:amalgamation}
Let $d\geq1$ be an integer, let $A$, $B$, and $Q$ be finite perfect groups, and let
$$
\alpha:A\twoheadrightarrow Q,\qquad\beta:B\twoheadrightarrow Q
$$
be epimorphisms. Suppose that $N=\ker\beta$ is a minimal normal subgroup of $B$. If $d(A),d(B)\leq d$, then there is a finite perfect group $C$ with $d(C)\leq d$ and epimorphisms $\pi_A:C\twoheadrightarrow A$ and $\pi_B:C\twoheadrightarrow B$ such that $\alpha\circ\pi_A=\beta\circ\pi_B$.
\end{lemma}

\begin{proof}
Form the fibre product
$$
E=A\times_QB=\{(a,b)\in A\times B: \alpha(a)=\beta(b)\}.
$$
We distinguish the central and noncentral cases.

Suppose first that $N\leq Z(B)$. We identify $N$ with $1\times N\leq E$; it is central in $E$, and $E/N\cong A$. Since $A$ is perfect, the image of $E'$ in $E/N$ is all of $E/N$, so
$$
E=E'N.
$$
The centrality of $N$ now gives
$$
E'=[E,E]=[E'N,E'N]=[E',E']=E''.
$$
Thus $C:=E'$ is perfect. Both coordinate projections of $E$ are surjective, and the image of a derived group under an epimorphism is the derived group of the image. Since $A$ and $B$ are perfect, the restrictions
$$
E'\twoheadrightarrow A,\qquad E'\twoheadrightarrow B
$$
are therefore surjective. The kernel of $E'\twoheadrightarrow A$ is contained in the central group $N$. Lemma \ref{lem:central-frattini} then yields
$$
d(E')=d(A)\leq d.
$$
Suppose now that $N\not\leq Z(B)$. Choose a generating $d$-tuple $(q_1,\ldots,q_d)$ of $Q$, padding with identity elements if necessary. Gasch\"utz's lifting lemma \cite{Gaschutz} gives generating tuples $(a_1,\ldots,a_d)$ of $A$ and $(b_1,\ldots,b_d)$ of $B$ such that
$$
\alpha(a_i)=q_i=\beta(b_i)\qquad(1\leq i\leq d).
$$
Put
$$
D=\langle(a_1,b_1),\ldots,(a_d,b_d)\rangle\leq E.
$$
Both coordinate projections of $D$ are surjective. Set
$$
K=D\cap(1\times N).
$$
Under the identification $1\times N\cong N$, the group $K$ is normal in $B$: it is normal in $D$, and the second projection $D\twoheadrightarrow B$ is surjective. The minimal normality of $N$ therefore gives $K=1$ or $K=N$.

If $K=1$, then $D\twoheadrightarrow A$ is an isomorphism. Hence $D$ is a perfect $d$-generated common cover of $A$ and $B$. If $K=N$, then $D=E$. Since $N$ is minimal normal and noncentral in $B$, one has
$$
[N,B]=N.
$$
The action of $E$ on $N$ is the action induced by its projection onto $B$, so $[N,E]=N$. Moreover $E/N\cong A$ is perfect. It follows that $E=E'N$ and $N\leq E'$, whence $E=E'$. Thus in this case $D=E$ is again a perfect $d$-generated common cover.
\end{proof}

The same argument applies to finitely many central extensions of a fixed perfect group.

\begin{corollary}\label{cor:central-packets}
Let $Q$ be a finite perfect group and let
$$
\beta_i:B_i\twoheadrightarrow Q\qquad(1\leq i\leq s)
$$
be finite perfect central extensions. Form their fibre product over $Q$,
$$
E=B_1\times_Q\cdots\times_QB_s.
$$
Then $E'$ is a finite perfect common cover of the $B_i$, and
$$
d(E')=d(Q).
$$
\end{corollary}

\begin{proof}
The kernel $Z$ of $E\twoheadrightarrow Q$ is central. Since $Q$ is perfect, $E=E'Z$, and therefore
$$
E'=[E'Z,E'Z]=[E',E'].
$$
Each projection $E\twoheadrightarrow B_i$ is surjective, so it maps $E'$ onto $B_i'=B_i$. Finally, the kernel of $E'\twoheadrightarrow Q$ is central in the perfect group $E'$. Lemma \ref{lem:central-frattini} gives $d(E')=d(Q)$.
\end{proof}
\begin{remark}\label{remark: interpret}
For a monolithic primitive group $L$ and a profinite group $\Gp$, put
$$
\kappa_L(\Gp)=\sup\{t\geq1:\Gp\twoheadrightarrow L_t\},\qquad\sup\varnothing=0.
$$
By Proposition \ref{prop:exact-crown-filtration}, a finite group $G$ satisfies $d(G)\leq d$ if and only if $\kappa_L(G)<f_L(d)$ for every $L$, when $d\geq2$.

In the central case of Lemma \ref{lem:amalgamation}, the groups $C=E'$ and $A$ have the same primitive crown quotients. Indeed, $\Frat(L)=1$ for every monolithic primitive group $L$. The coordinate maps $L_t\twoheadrightarrow L$ have trivial common kernel, so $\Frat(L_t)=1$ as well. Since $\ker(C\twoheadrightarrow A)\leq\Frat(C)$, every epimorphism $C\twoheadrightarrow L_t$ factors through $A$. Thus
$$
\kappa_L(C)=\kappa_L(A)
$$
for every $L$.

In the noncentral case, the subgroup $D$ generated by the lifted $d$-tuple either maps isomorphically onto $A$ or equals $A\times_QB$. Both cases give a perfect common cover with at most $d$ generators. The crown criterion therefore bounds all its crown multiplicities by $f_L(d)-1$. Proposition \ref{prop:crowns-of-Ud} describes the multiplicities in the limit: each perfect monolithic primitive group $L$ has crown multiplicity $f_L(d)-1$.
\end{remark}

\begin{proposition}\label{prop:directedness}
Let $d\geq1$ be an integer. Any two finite perfect $d$-generated groups have a finite perfect $d$-generated common cover.
\end{proposition}

\begin{proof}
Let $A$ and $B$ be finite perfect groups with $d(A),d(B)\leq d$. Choose a chief series
$$
1=N_0<N_1<\cdots<N_r=B.
$$
Thus $N_{j+1}/N_j$ is a minimal normal subgroup of $B/N_j$ for $0\leq j<r$. Every quotient $B/N_j$ is perfect.

We descend from $B/N_r=1$ to $B/N_0=B$. Start with $C_r=A$, equipped with the unique epimorphism $C_r\twoheadrightarrow B/N_r$. Suppose that a finite perfect $d$-generated group $C_{j+1}$ has been constructed which maps onto $A$ and onto $B/N_{j+1}$. Apply Lemma \ref{lem:amalgamation} to
$$
C_{j+1}\twoheadrightarrow B/N_{j+1}\qquad\text{and}\qquad B/N_j\twoheadrightarrow B/N_{j+1}.
$$
It gives a finite perfect $d$-generated group $C_j$ mapping onto both $C_{j+1}$ and $B/N_j$. At the end, $C_0$ maps onto both $A$ and $B$. This proves the assertion for two groups; induction gives the assertion for any finite family.
\end{proof}

\subsection{Existence of the universal group}

\begin{proof}[Proof of Theorem \ref{thm:nikolov-B1-full}\textup{(i)} and \textup{(ii)}]
Part (i) is Proposition \ref{prop:directedness}.

For part (ii), enumerate the isomorphism classes of finite perfect $d$-generated groups as
$$
H_1,H_2,\ldots.
$$
Start with $P_1=H_1$. Using Proposition \ref{prop:directedness} recursively, construct finite perfect $d$-generated groups $P_n$ and epimorphisms
$$
P_{n+1}\twoheadrightarrow P_n,\qquad P_n\twoheadrightarrow H_n.
$$
By composition, $P_n$ maps onto every $H_i$ with $i\leq n$. Put
$$
U_d=\varprojlim_nP_n.
$$
The projection $U_d\twoheadrightarrow P_n$ is surjective, so every $H_n$ is a continuous quotient of $U_d$.

Every finite continuous quotient of $U_d$ factors through some $P_n$, since the kernels of the projections form a basis of open neighbourhoods of the identity. It is therefore perfect and $d$-generated. Proposition \ref{prop:finite-quotient-detects-generation} gives $d(U_d)\leq d$. Thus the finite continuous quotients of $U_d$ are precisely the finite perfect $d$-generated groups.

Finally, since every finite continuous quotient of $U_d$ is perfect, the abstract derived subgroup $[U_d,U_d]$ is dense in $U_d$. The derived subgroup of a topologically finitely generated profinite group is closed by the uniform commutator theorem of Nikolov and Segal \cite[Theorem 1.4]{NikolovSegalI}. Therefore
$$
U_d=[U_d,U_d],
$$
as required.
\end{proof}

\begin{proposition}\label{prop:crowns-of-Ud}
Let $d\geq2$, and let $U_d$ be a profinite group whose finite continuous quotients are precisely the finite perfect groups $H$ with $d(H)\leq d$. For every monolithic primitive finite group $L$,
$$
\kappa_L(U_d)=\begin{cases}
f_L(d)-1,&\text{if $L$ is perfect},\\
0,&\text{otherwise}.
\end{cases}
$$
If $U_d=\varprojlim_n P_n$ is an inverse limit of finite groups with surjective transition maps, then for each fixed $L$ there is an integer $n_L$ such that
$$
\kappa_L(P_n)=\kappa_L(U_d)\qquad(n\geq n_L).
$$
\end{proposition}

\begin{proof}
Put $A=\soc(L)$. We first show that $L_t$ is perfect if and only if $L$ is perfect. One implication follows from the projection $L_t\twoheadrightarrow L$. Conversely, suppose that $L$ is perfect. The subgroup $A$ cannot be central: Lemma \ref{lem:central-frattini} would give $A\leq\Frat(L)$, contrary to primitivity. Minimal normality therefore gives $[A,L]=A$. The projection of $L_t$ onto each coordinate is surjective, so the commutators with each coordinate subgroup of $A^t$ generate that subgroup. Thus $A^t\leq[L_t,L_t]$. Since $L_t/A^t\cong L/A$ is perfect, so is $L_t$.

If $L$ is not perfect, no $L_t$ is a quotient of $U_d$, and hence $\kappa_L(U_d)=0$. If $L$ is perfect, the defining property of $U_d$ gives
$$
U_d\twoheadrightarrow L_t\quad\Longleftrightarrow\quad d(L_t)\leq d
\quad\Longleftrightarrow\quad t<f_L(d).
$$
This proves the formula, including the case $f_L(d)=1$.

Put $k=\kappa_L(U_d)$. Each $P_n$ is a quotient of $U_d$, so $\kappa_L(P_n)\leq k$. If $k=0$, equality holds for every $n$. If $k>0$, an epimorphism $U_d\twoheadrightarrow L_k$ factors through some $P_{n_L}$, and hence through every $P_n$ with $n\geq n_L$. Thus $\kappa_L(P_n)\geq k$ for those $n$, proving the assertion.
\end{proof}

\subsection{Uniqueness}\label{subsec:uniqueness}
A topologically finitely generated profinite group is determined by its finite continuous quotients. We include a proof of this profinite form of the theorem of Dixon, Formanek, Poland, and Ribes \cite{DFPR}.

\begin{lemma}\label{lem:hopfian}
Let $\Gp$ be a topologically finitely generated profinite group. Then every continuous epimorphism $\varphi:\Gp\twoheadrightarrow\Gp$ is an isomorphism.
\end{lemma}

\begin{proof}
Fix $n\geq1$. A continuous action of $\Gp$ on $\{1,\ldots,n\}$ is determined by the images of a finite topological generating set. Hence there are finitely many such actions. Every open subgroup of index $n$ is a point stabilizer of one of them, so the set $\mathcal U_n$ of these subgroups is finite. Since $\varphi$ is surjective, the map $U\mapsto\varphi^{-1}(U)$ is an injection from $\mathcal U_n$ to itself and hence a bijection. Every open subgroup of index $n$ therefore contains $\ker\varphi$. Since the open subgroups intersect in the identity, $\ker\varphi=1$. Recall that for compact Hausdorff groups it suffices to show that a continuous homomorphism $\varphi$ is bijective to grant that it is an isomorphism of topological groups. 
\end{proof}

\begin{proposition}\label{prop:rigidity}
Let $\Gp$ and $\mathcal H$ be topologically finitely generated profinite groups such that a finite group is a continuous quotient of $\Gp$ if and only if it is a continuous quotient of $\mathcal H$. Then $\Gp\cong\mathcal H$ as topological groups.
\end{proposition}

\begin{proof}
For an open normal subgroup $M\trianglelefteq\mathcal H$, put
$$
S_M=\Sur(\Gp,\mathcal H/M).
$$
Each $S_M$ is nonempty by hypothesis. It is finite because a continuous homomorphism to a finite group is determined by the images of a finite topological generating set. Postcomposition with the projections $\mathcal H/M\twoheadrightarrow\mathcal H/M'$ for $M\leq M'$ makes $(S_M)_M$ an inverse system of nonempty finite sets over the directed set of open normal subgroups of $\mathcal H$. Its inverse limit is nonempty by compactness. An element is a compatible family of epimorphisms, that is, a continuous homomorphism
$$
\varphi:\Gp\longrightarrow\varprojlim_M\mathcal H/M=\mathcal H
$$
whose image is dense. The image is compact and hence closed, so $\varphi$ is surjective. The same argument gives a continuous epimorphism $\psi:\mathcal H\twoheadrightarrow\Gp$. The composite $\psi\circ\varphi:\Gp\twoheadrightarrow\Gp$ is an isomorphism by Lemma \ref{lem:hopfian}. Thus $\varphi$ is injective and hence an isomorphism.
\end{proof}

\begin{proof}[Proof of Theorem \ref{thm:nikolov-B1-full}\textup{(iii)}]
Let $\mathcal V$ be as in (iii). Every finite continuous quotient $Q$ of $\mathcal V$ is $d$-generated. Since $\mathcal V=\overline{[\mathcal V,\mathcal V]}$, its image in $Q$ is $\overline{[Q,Q]}=[Q,Q]$, so $Q$ is perfect. By hypothesis the finite continuous quotients of $\mathcal V$ also include every finite perfect $d$-generated group. Hence $\mathcal V$ and $U_d$ have the same finite continuous quotients, and Proposition \ref{prop:rigidity} gives $\mathcal V\cong U_d$.
\end{proof}

\subsection{Projectivity and presentations}\label{subsec:projectivity}
A profinite group $P$ is projective if every continuous homomorphism $P\to B$ lifts through every continuous epimorphism $A\twoheadrightarrow B$ of profinite groups.

\begin{proposition}\label{prop:Ud-projective}
Let $d\geq2$. Every continuous epimorphism $\wideF_d\twoheadrightarrow U_d$ has a continuous section. In particular, $U_d$ is projective.
\end{proposition}

\begin{proof}
Fix an epimorphism $\pi:F:=\wideF_d\twoheadrightarrow U_d$. There is a closed subgroup $P\leq F$ minimal among those mapping onto $U_d$. Indeed, the intersection of a descending chain of such subgroups still maps onto $U_d$: its intersection with each fibre of $\pi$ is nonempty by compactness. Zorn's lemma therefore applies.

Choose a generating $d$-tuple of $U_d$ and lift it to $P$. The closed subgroup generated by these lifts maps onto $U_d$, so it equals $P$ by minimality. Thus $P$ is topologically $d$-generated. The closed subgroup $\overline{[P,P]}$ also maps onto $U_d$, since $U_d$ is perfect. Minimality gives
$$
P=\overline{[P,P]}.
$$
Since $P$ maps onto $U_d$, its finite quotients include every finite perfect $d$-generated group. Theorem \ref{thm:nikolov-B1-full}(iii) gives an isomorphism $j:U_d\to P$. The composite $\pi|_P\circ j$ is a surjective endomorphism of $U_d$, hence an isomorphism by Lemma \ref{lem:hopfian}. Therefore $\pi|_P$ is an isomorphism, and its inverse gives a continuous section $s:U_d\to F$.

To prove projectivity, let $\alpha:A\twoheadrightarrow B$ and $f:U_d\to B$ be continuous. Freeness gives a lift $\ell:F\to A$ of $f\circ\pi$, by lifting the images of the free generators. Then $\ell\circ s$ lifts $f$.
\end{proof}

\begin{lemma}\label{lem:idempotent-relators}
Let $F=\wideF_d$ have free generators $x_1,\ldots,x_d$, and let $e:F\to F$ be a continuous endomorphism with $e^2=e$. Then
$$
\ker e=\overline{\langle x_1e(x_1)^{-1},\ldots,x_de(x_d)^{-1}\rangle^F}.
$$
\end{lemma}

\begin{proof}
Let $R$ be the closed normal subgroup on the right. Each of its stated generators lies in $\ker e$, so $R\leq\ker e$. If $q:F\twoheadrightarrow F/R$ is the quotient map, then $q=q\circ e$ on the free generators, and hence on $F$. Thus $q$ kills $\ker e$, proving the reverse inclusion.
\end{proof}

\begin{proposition}\label{prop:balanced-presentation}
Let $d\geq2$. Then $d(U_d)=d$, and $U_d$ admits a profinite presentation with $d$ generators and $d$ relators. Every profinite presentation of $U_d$ has at least $d$ relators.
\end{proposition}

\begin{proof}
We already know that $d(U_d)\leq d$. For the reverse inequality, it suffices to find a finite perfect quotient of rank $d$. For $d=2$, take $A_5$. For $d\geq3$, put $t=f_{A_5}(d-1)$. Appending an arbitrary element of $A_5$ to a generating $m$-tuple gives
$$
|\Gen_{m+1}(A_5)|\geq60|\Gen_m(A_5)|>|\Gen_m(A_5)|\qquad(m\geq2).
$$
By \eqref{eq:simple-fLm}, this implies $f_{A_5}(d-1)<f_{A_5}(d)$. Hence
$$
d-1<d(A_5^t)\leq d.
$$
The finite perfect group $A_5^t$ is therefore a quotient of $U_d$ of rank $d$.

Choose an epimorphism $\pi:F:=\wideF_d\twoheadrightarrow U_d$. Proposition \ref{prop:Ud-projective} gives a section $s:U_d\to F$. Then $e=s\circ\pi$ is idempotent and $\ker e=\ker\pi$. Lemma \ref{lem:idempotent-relators} gives the required $d$ relators.

For minimality, it suffices to consider finitely many relators. Suppose first that the presentation has $n<\infty$ generators and relators $y_1,\ldots,y_k\in\wideF_n$. For any prime $p$, let $\bar y_i$ be their images in $\F_p^n$. Since $U_d$ is perfect, passing to the maximal elementary abelian $p$-quotient gives
$$
0\cong\F_p^n/\langle\bar y_1,\ldots,\bar y_k\rangle_{\F_p},
$$
so $k\geq n\geq d(U_d)=d$. If the free source has infinitely many generators, it maps onto $C_p^{k+1}$. The images of the $k$ relators cannot span this quotient. The presented group then has a nontrivial elementary abelian quotient, so it cannot be $U_d$.
\end{proof}

\begin{proof}[Proof of Theorem \ref{thm:nikolov-B1-full}\textup{(iv)}]
Combine Propositions \ref{prop:Ud-projective} and \ref{prop:balanced-presentation}.
\end{proof}

\begin{remark}\label{rem:retract}
The section in Proposition \ref{prop:Ud-projective} identifies $U_d$ with a closed subgroup of $\wideF_d$, and $s\circ\pi$ is a retraction onto this subgroup.
\end{remark}

\begin{remark}\label{rem:what-is-settled}
Theorem \ref{thm:nikolov-B1-full} proves the profinite Conjecture B1 of \cite{NikolovFiniteImages}. By Nikolov's reductions, Conjecture A of \cite{NikolovFiniteImages} is therefore equivalent to Conjecture B2. The latter asks for a dense finitely generated abstractly perfect subgroup in every finitely generated perfect profinite group. Every topologically $d$-generated perfect profinite group has all its finite continuous quotients among those of $U_d$. The compactness argument of Proposition \ref{prop:rigidity} therefore makes it a continuous quotient of $U_d$. Dense finitely generated abstractly perfect subgroups pass to quotients. Thus Conjecture A is equivalent to the assertion that $U_d$ contains such a subgroup for every $d\geq2$.
\end{remark}

\section{Further applications}\label{sec:further-applications}
We apply the moment bounds to random groups with finite group actions and study positive finite generation in the Liu--Wood model.

\subsection{A transfer principle for finite group actions}\label{subsec:equivariant-transfer}
Let $\Gamma$ be a finite group. A profinite $\Gamma$-group is a profinite group with a continuous action of $\Gamma$ by automorphisms. We write $\Sur_\Gamma(X,A)$ for the set of continuous epimorphisms $\phi:X\to A$ such that $\phi(\gamma(x))=\gamma(\phi(x))$ for all $x\in X$ and $\gamma\in\Gamma$. For $\Delta\leq\Gamma$, put
$$
A^\Delta:=\{a\in A:\gamma(a)=a\text{ for every }\gamma\in\Delta\}.
$$
The semidirect product $X\rtimes\Gamma$ has the product topology and multiplication
$$
(x,\gamma)(y,\eta)=(x\gamma(y),\gamma\eta).
$$

Let $\Pi$ be a set of primes not dividing $|\Gamma|$. A finite $\Pi$-group is a finite group whose order has all its prime divisors in $\Pi$. A profinite group is pro-$\Pi$ if every finite continuous quotient is a $\Pi$-group. We also write pro-$m'$ when $\Pi$ consists of the primes not dividing $m$. A pro-$\Pi$ $\Gamma$-group $X$ is \emph{admissible} if
$$
X=\overline{\langle x^{-1}\gamma(x):x\in X,\ \gamma\in\Gamma\rangle}.
$$
For a finite group $Q$, let $O_\Pi(Q)$ be its largest normal $\Pi$-subgroup.

\begin{proposition}\label{prop:transfer}
Let $X$ be a random admissible pro-$\Pi$ $\Gamma$-group. Suppose that the functions $|\Sur_\Gamma(X,A)|$ are measurable and that, for some $M\geq1$ and $\beta\geq0$,
\begin{equation}\label{eq:equivariant-moment}
\E|\Sur_\Gamma(X,A)|\leq M|A|^\beta
\end{equation}
for every finite admissible $\Pi$-$\Gamma$-group $A$. Then, for every finite group $Q$,
\begin{equation}\label{eq:transferred-moment}
\E|\Sur(X\rtimes\Gamma,Q)|\leq M|Q|^{d(\Gamma)+\beta}.
\end{equation}
Both $X\rtimes\Gamma$ and $X$ are topologically finitely generated almost surely.
\end{proposition}

\begin{proof}
Put $Y=X\rtimes\Gamma$. If $f:Y\twoheadrightarrow Q$ and $N=f(X)$, then $N$ is a normal $\Pi$-subgroup and $Q/N$ is a quotient of $\Gamma$. Every normal $\Pi$-subgroup of $Q$ maps trivially to $Q/N$, since its image has order dividing $|\Gamma|$ and has all its prime divisors in $\Pi$. Hence $N=O_\Pi(Q)$.

Let $\rho=f|_\Gamma$. The restriction $f|_X$ is a $\Gamma$-equivariant epimorphism onto $O_\Pi(Q)$ for the action
$$
\gamma\cdot a=\rho(\gamma)a\rho(\gamma)^{-1}.
$$
Denote this $\Gamma$-group by $O_\Pi(Q)_\rho$. Conversely, if $\rho:\Gamma\to Q$ satisfies $Q=O_\Pi(Q)\rho(\Gamma)$ and $\phi:X\twoheadrightarrow O_\Pi(Q)_\rho$ is equivariant, then
$$
(x,\gamma)\longmapsto\phi(x)\rho(\gamma)
$$
is an epimorphism $Y\twoheadrightarrow Q$. Thus
\begin{equation}\label{eq:transfer-identity}
|\Sur(Y,Q)|=\sum_{\substack{\rho\in\Hom(\Gamma,Q)\\Q=O_\Pi(Q)\rho(\Gamma)}}|\Sur_\Gamma(X,O_\Pi(Q)_\rho)|.
\end{equation}
Every equivariant quotient of $X$ is admissible. A summand in \eqref{eq:transfer-identity} is therefore zero unless its target is admissible. In the remaining cases, \eqref{eq:equivariant-moment} applies. Since
$$
|\Hom(\Gamma,Q)|\leq |Q|^{d(\Gamma)},\qquad |O_\Pi(Q)|\leq |Q|,
$$
taking expectations gives \eqref{eq:transferred-moment}. The finite sum also proves measurability. Theorem \ref{thm:finite-generation}, with $\delta=0$, shows that $Y$ is finitely generated almost surely. The profinite Schreier bound \cite[Section 3.6]{RibesZalesskii} gives
$$
d(X)\leq 1+[Y:X](d(Y)-1)=1+|\Gamma|(d(Y)-1),
$$
so $X$ is finitely generated almost surely.
\end{proof}

We apply this proposition to three probability measures defined by random presentations. We write $\mu_{\Gamma,u}$ for the measure constructed by Liu--Wood--Zureick-Brown in \cite[Definition 3.15 and Section 5.4]{LWZB}, $\mu_{\Gamma,\Gamma_\infty}$ for the measure constructed by Liu--Willyard in \cite[Definition 6.1]{LiuWillyard}, and $\mu_{\boldsymbol\Gamma}$ for the measure constructed by Willyard in \cite[Definitions 5.1, 5.2, and 5.5]{Willyard}. Here $u\in\Z$ in the first case, $\Gamma_\infty\leq\Gamma$ is cyclic in the second, and $\boldsymbol\Gamma=(\Gamma_1,\ldots,\Gamma_{u+1})$, with $u\geq0$ and $\Gamma_i\leq\Gamma$, in the third. The underlying spaces and their topologies are defined in \cite[Section 3.2]{LWZB}, \cite[Section 6]{LiuWillyard}, and \cite[Section 5]{Willyard}, respectively. Their basic open sets specify the isomorphism type of the maximal quotient in the class generated by a finite set of finite $\Gamma$-groups under $\Gamma$-invariant subgroups, $\Gamma$-equivariant quotients, and finite products. The functions $|\Sur_\Gamma(X,A)|$ are constant on such basic open sets once the chosen class contains $A$, and hence are measurable.

\begin{corollary}\label{cor:arithmetic-laws}
The following random profinite $\Gamma$-groups are topologically finitely generated almost surely.
\begin{enumerate}[label=\textup{(\roman*)}]
\item The Liu--Wood--Zureick-Brown law $\mu_{\Gamma,u}$, for $u\in\Z$.
\item The Liu--Willyard law $\mu_{\Gamma,\Gamma_\infty}$, where $\Gamma_\infty\leq\Gamma$ is cyclic.
\item Willyard's law $\mu_{\boldsymbol\Gamma}$, where $\boldsymbol\Gamma=(\Gamma_1,\ldots,\Gamma_{u+1})$, $u\geq0$, and $\Gamma_i\leq\Gamma$.
\end{enumerate}
If $X$ has the respective law, then for every finite group $Q$,
$$
\E|\Sur(X\rtimes\Gamma,Q)|\leq
\begin{cases}
|Q|^{d(\Gamma)+\max\{0,-u\}}&\text{in \textup{(i)}},\\
|Q|^{d(\Gamma)}&\text{in \textup{(ii)} and \textup{(iii)}}.
\end{cases}
$$
\end{corollary}

\begin{proof}
These laws are supported on admissible pro-$|\Gamma|'$ groups in \textup{(i)} and \textup{(iii)}, and on admissible pro-$(2|\Gamma|)'$ groups in \textup{(ii)}. For each finite admissible target $A$ in the corresponding category, the cited papers prove
\begin{align}
\E_{\mu_{\Gamma,u}}|\Sur_\Gamma(X,A)|
&=[A:A^\Gamma]^{-u}\leq |A|^{\max\{0,-u\}},\label{eq:LWZB-moment}\\
\E_{\mu_{\Gamma,\Gamma_\infty}}|\Sur_\Gamma(X,A)|
&=[A^{\Gamma_\infty}:A^\Gamma]^{-1}\leq1,\label{eq:LiuWillyard-moment}\\
\E_{\mu_{\boldsymbol\Gamma}}|\Sur_\Gamma(X,A)|
&=\frac{|A^\Gamma|}{\prod_{i=1}^{u+1}|A^{\Gamma_i}|}\leq |A^\Gamma|^{-u}\leq1.\label{eq:Willyard-moment}
\end{align}
These are \cite[Theorem 6.2]{LWZB}, \cite[Theorem 6.3(3)]{LiuWillyard}, and \cite[Theorem 6.2]{Willyard}, respectively. The last inequality uses $A^\Gamma\leq A^{\Gamma_i}$ for each $i$. Proposition \ref{prop:transfer} gives the assertions.
\end{proof}

The corollary concerns the probability measures defined by these random presentations. Their conjectural relation to Galois groups is described in the cited papers.

\subsection{Positive finite generation}\label{subsec:PFG}
For a profinite group $\Gp$, let $\mu_{\Gp}$ be its normalized Haar measure, the translation-invariant Borel probability measure on $\Gp$. Put
$$
P(\Gp,d):=\mu_{\Gp}^{\otimes d}\{(g_1,\ldots,g_d):\overline{\langle g_1,\ldots,g_d\rangle}=\Gp\}.
$$
The group $\Gp$ is \emph{positively finitely generated}, or PFG, if $P(\Gp,d)>0$ for some integer $d\geq1$. It has \emph{polynomial maximal-subgroup growth} if the number of maximal open subgroups of index $n$ is at most $n^c$ for some fixed $c$ and every $n\geq2$. These properties are equivalent by \cite[Theorem 4]{MannShalev}.

A closed normal subgroup $N\trianglelefteq H$ is \emph{positively finitely normally generated in $H$} if, for some $k\geq1$,
$$
\mu_N^{\otimes k}\{(x_1,\ldots,x_k):\overline{\langle x_1,\ldots,x_k\rangle^H}=N\}>0.
$$
Here $\overline{\langle x_1,\ldots,x_k\rangle^H}$ is the smallest closed normal subgroup of $H$ containing the tuple. A finitely generated profinite group $\Gp$ is \emph{positively finitely related}, or PFR, if this condition holds for the kernel of every continuous epimorphism $H\twoheadrightarrow\Gp$ with $H$ finitely generated \cite[Definition 3.3 and Lemma 3.4]{KionkeVannacci}.

For a prime $p$ and $n\geq1$, let $r_n(\Gp,p)$ count the isomorphism classes of continuous irreducible representations $\Gp\to\GL_n(\F_p)$. Irreducibility means that the only invariant subspaces are $0$ and $\F_p^n$; isomorphism means conjugacy by $\GL_n(\F_p)$. The group $\Gp$ has \emph{uniformly bounded exponential representation growth}, or UBERG, if
$$
r_n(\Gp,p)\leq p^{cn}
$$
for some $c$ independent of $p$ and $n$ \cite[Definition 5.1]{KionkeVannacci}.

A \emph{minimal extension} of a finitely generated profinite group $\Gp$ is a continuous epimorphism $\pi:E\twoheadrightarrow\Gp$, with $E$ profinite, whose kernel is a finite nontrivial minimal normal subgroup of $E$. Its degree is $|\ker\pi|$. We count extensions up to continuous isomorphisms commuting with their maps to $\Gp$. The group $\Gp$ has \emph{polynomial minimal-extension growth} if there is a constant $c$ such that the number of these isomorphism classes of degree $n$ is at most $n^c$ for every $n\geq2$ \cite[Section 3.4]{KionkeVannacci}.

A profinite group $\Gp$ is \emph{positively finitely presented} if it is PFG and the kernel of every epimorphism $P\twoheadrightarrow\Gp$ from a PFG projective profinite group is positively finitely normally generated in $P$. Projectivity is defined in Section \ref{sec:perfect-universal}. Equivalently, $\Gp$ is PFG and finitely presented \cite[Definition 3.2 and Proposition 3.3]{CorobCookVannacci}.

We use the following criterion of Jaikin-Zapirain and Pyber. Let $L$ be a finite monolithic primitive group with nonabelian socle $A$. Recall its crown multiplicity
\begin{equation}\label{eq:def-kappa}
\kappa_L(\Gp):=\sup\{k\geq1:\Gp\twoheadrightarrow L_k\},
\end{equation}
where the supremum of the empty set is $0$. Write $l(A)$ for the least degree of a faithful transitive permutation representation of $A$. Thus $l(A)$ is the least cardinality of a finite set on which $A$ acts transitively with trivial kernel. A transitive permutation action is \emph{primitive} if it preserves no partition other than the partition into singletons and the partition with one part.

A finitely generated profinite group $\Gp$ is PFG if and only if there is a constant $c$ such that
\begin{equation}\label{eq:JZP-criterion}
\kappa_L(\Gp)\leq l(A)^c
\end{equation}
for every finite monolithic primitive group $L$ with nonabelian socle $A$. An equivalent condition is that, for some $c$,
\begin{equation}\label{eq:JZP-epi-criterion}
|\Sur(\Gp,L)|\leq |L|\,l(A)^c
\end{equation}
for every such $L$ \cite[Theorem 11.1]{JaikinPyber}.

The properties used below define Borel sets on the locus of finitely generated groups. Indeed, this locus is Borel by Section \ref{sec:witnesses}, and \eqref{eq:JZP-epi-criterion} expresses PFG using countably many inequalities between continuous epimorphism-counting functions. For fixed $p,n$, the function $r_n(\Gp,p)$ depends only on the level-$\{\GL_n(\F_p)\}$ completion, so it is continuous. Taking the union over integral $c\geq1$ in the definition of UBERG shows that its locus is Borel. Finally, PFR and polynomial minimal-extension growth are equivalent to finite presentation together with UBERG on this locus \cite[Theorems 3.9 and 5.6]{KionkeVannacci}. Proposition \ref{prop:SW-measurability} and the characterization of positive finite presentation above give the remaining assertions.

\begin{lemma}\label{lem:PFG-monolith-count}
For each integer $D\geq1$, there is a constant $C_D$ such that, for every integer $n\geq2$, at most $n^{C_D}$ isomorphism classes of finite monolithic primitive groups $L$ satisfy
$$
d(L)\leq D,\qquad A=\soc(L)\text{ is nonabelian},\qquad l(A)=n.
$$
\end{lemma}

\begin{proof}
By \cite[Lemma 9.2]{JaikinPyber}, each such $L$ has a faithful primitive permutation representation of degree at most $n^{c_7}$, for an absolute constant $c_7$. By \cite[Theorem 8.1]{JaikinPyber}, there are at most $m^{c_pD}$ conjugacy classes of $D$-generated primitive subgroups of $\Sym(m)$, for an absolute constant $c_p$. Thus the number in question is at most
\begin{equation}\label{eq:CD-count}
\sum_{m\leq n^{c_7}}m^{c_pD}\leq n^{c_7(c_pD+1)}.
\end{equation}
We may take $C_D=c_7(c_pD+1)$.
\end{proof}

For a nonabelian finite simple group $S$, the quantity $Q_S(\Gp)$ from \eqref{eq:def-QS} counts the open normal subgroups with quotient isomorphic to $S$. Distinct such kernels give an epimorphism onto the product of the quotients. Hence
\begin{equation}\label{eq:kappa-equals-Q}
\kappa_S(\Gp)=Q_S(\Gp).
\end{equation}
Lemma \ref{lem:simple-poisson} gives their joint distribution under $\mu_u$.

\begin{theorem}\label{thm:LW-PFG}
Let $u\in\Z$, and let $\Gp$ have law $\mu_u$.
\begin{enumerate}[label=\textup{(\roman*)}]
\item If $u\geq0$, then almost surely
\begin{equation}\label{eq:bounded-crowns}
\sup_L\kappa_L(\Gp)<\infty,
\end{equation}
where $L$ ranges over the finite monolithic primitive groups with nonabelian socle.
\item If $u=-1$, then $\Gp$ is PFG almost surely, but
\begin{equation}\label{eq:unbounded-crowns}
\sup_L\kappa_L(\Gp)=\infty
\end{equation}
almost surely. For every $k\geq1$, infinitely many alternating simple groups $A_n$ satisfy $\Gp\twoheadrightarrow A_n^k$, almost surely.
\item If $u\leq-2$, then $\Gp$ is not PFG and does not have UBERG, almost surely.
\end{enumerate}
Consequently,
\begin{equation}\label{eq:PFG-dichotomy}
\mu_u\{\Gp:\Gp\text{ is PFG}\}=\begin{cases}
1,&u\geq-1,\\
0,&u\leq-2.
\end{cases}
\end{equation}
In the range $u\geq-1$, almost every $\Gp$ is PFR, has polynomial minimal-extension growth, has UBERG, and is positively finitely presented. In the range $u\leq-2$, almost every $\Gp$ has none of these properties.
\end{theorem}

\begin{proof}
Suppose first that $u\geq0$, and fix $D\geq1$. Let $L$ be monolithic primitive with nonabelian socle $A$, and put $K=L/A$ and $n=l(A)$. The action of $A$ on itself by left multiplication gives $n\leq|A|$. Since $|L_k|=|K||A|^k$, Proposition \ref{prop:crown-automorphisms}(ii) and the Liu--Wood moment identity give
\begin{equation}\label{eq:one-crown-PFG}
\mu_u\{\Gp:\Gp\twoheadrightarrow L_k\}
\leq\frac{|L_k|^{-u}}{|\Aut(L_k)|}
\leq\frac{|K|^{-u}|A|^{-(u+1)k}}{k!}
\leq n^{-(u+1)k}.
\end{equation}
Let $E_{D,k}$ be the event that $d(\Gp)\leq D$ and $\Gp$ has a quotient $L_k$ with $\soc(L)$ nonabelian. Every such $L$ is $D$-generated. Lemma \ref{lem:PFG-monolith-count} therefore gives
\begin{equation}\label{eq:EDk-bound}
\mu_u(E_{D,k})\leq\sum_{n\geq2}n^{C_D-(u+1)k}.
\end{equation}
For sufficiently large $k$, the summands are bounded by $n^{-2}$ and tend to zero termwise as $k\to\infty$. Hence the right-hand side tends to zero by dominated convergence. The events $E_{D,k}$ decrease, with intersection
$$
\{d(\Gp)\leq D,\ \sup_L\kappa_L(\Gp)=\infty\}.
$$
This intersection has measure zero. Almost-sure finite generation and a union over $D$ prove \textup{(i)}.

We next prove that $\Gp$ is PFG almost surely for every $u\geq-1$. Fix $D\geq1$, and choose an integer $a>C_D+1$. For each $D$-generated monolithic primitive group $L$ with nonabelian socle $A$, set $n=l(A)$ and
$$
B_L(a):=\{\Gp:|\Sur(\Gp,L)|>|L|n^a\}.
$$
Markov's inequality gives
\begin{equation}\label{eq:markov-PFG}
\mu_u(B_L(a))\leq |L|^{-u-1}n^{-a}\leq n^{-a}.
\end{equation}
By Lemma \ref{lem:PFG-monolith-count},
$$
\sum_{\substack{L\text{ monolithic primitive}\\ \soc(L)\text{ nonabelian},\ d(L)\leq D}}
\mu_u(B_L(a))\leq\sum_{n\geq2}n^{C_D-a}<\infty.
$$
The first Borel--Cantelli lemma shows that, almost surely, only finitely many of these $L$ violate \eqref{eq:JZP-epi-criterion} with exponent $a$. On the event $d(\Gp)\leq D$, all epimorphism counts are finite. The finitely many exceptions can therefore be included by increasing $a$ to a finite value depending on $\Gp$. A group $L$ with $d(L)>D$ is not a quotient of $\Gp$, so its epimorphism count is zero. Thus \eqref{eq:JZP-epi-criterion} holds for every relevant $L$. Almost-sure finite generation and a union over $D$ prove the assertion.

When $u=-1$, Lemma \ref{lem:simple-poisson} gives independent random variables
$$
Q_{A_n}\sim\operatorname{Pois}(1/2)\qquad(n\geq7),
$$
since $|\Out(A_n)|=2$. For each fixed $k\geq1$, the events $Q_{A_n}\geq k$ are independent and have the same positive probability. By the second Borel--Cantelli lemma, infinitely many of them occur, almost surely. Intersecting over $k$ and using \eqref{eq:kappa-equals-Q} proves \textup{(ii)}.

Suppose now that $u\leq-2$. For $m\geq5$, put
$$
T_m:=\mathrm{PSL}_m(2)=\GL_m(2).
$$
These are pairwise nonisomorphic finite simple groups with $|\Out(T_m)|\leq2$ \cite{Atlas}, and
\begin{equation}\label{eq:Tm-order}
|T_m|=2^{m^2}\prod_{j=1}^m(1-2^{-j}) \geq \eta_2(\infty)2^{m^2}.
\end{equation}
The Poisson mean in \eqref{eq:simple-Poisson-PFG} consequently satisfies
\begin{equation}\label{eq:Tm-mean}
\lambda_{T_m,u}=\frac{|T_m|^{-u-1}}{|\Out(T_m)|} \geq \frac{\eta_2(\infty)}2\,2^{m^2}.
\end{equation}
Fix $c\geq1$. For all sufficiently large $m$, $2^{cm}\leq\lambda_{T_m,u}/2$. Applying Markov's inequality to $2^{-Q_{T_m}}$ gives
\begin{equation}\label{eq:poisson-tail}
\mu_u\{Q_{T_m}\leq2^{cm}\} \leq \mu_u\{Q_{T_m}\leq\lambda_{T_m,u}/2\} \leq \exp(-\lambda_{T_m,u}/8).
\end{equation}
The right-hand side is summable in $m$. Hence, almost surely,
\begin{equation}\label{eq:QTm-large}
Q_{T_m}(\Gp)>2^{cm}
\end{equation}
for all sufficiently large $m$. Intersecting over integral $c$ makes this simultaneous for every fixed $c$.

The natural action of $T_m=\GL_m(2)$ on the nonzero vectors of $\F_2^m$ is faithful and transitive, and hence
$$
l(T_m) \leq 2^m-1.
$$
Together with \eqref{eq:kappa-equals-Q}, \eqref{eq:QTm-large} violates \eqref{eq:JZP-criterion} for every exponent. Thus $\Gp$ is not PFG.

The same groups prove failure of UBERG. The natural representation of $T_m=\GL_m(2)$ on $\F_2^m$ is faithful and irreducible. Choose one epimorphism $\Gp\twoheadrightarrow T_m$ for each kernel counted by $Q_{T_m}(\Gp)$, and compose it with this representation. The resulting representations have distinct kernels and hence are pairwise nonisomorphic. Thus
$$
r_m(\Gp,2)\geq Q_{T_m}(\Gp).
$$
Equation \eqref{eq:QTm-large} contradicts every bound $r_m(\Gp,2)\leq2^{cm}$, so $\Gp$ does not have UBERG.

By Theorem \ref{thm:finite-presentation}, $\Gp$ is finitely presented almost surely. PFR is equivalent to polynomial minimal-extension growth, and for finitely presented groups it is also equivalent to UBERG \cite[Theorems 3.9 and 5.6]{KionkeVannacci}. PFG implies UBERG \cite[Corollary 1.11]{CorobCookVannacci}. Together with the characterization of positive finite presentation above, these implications prove the remaining assertions.

\end{proof}

Theorem \ref{thm:LW-PFG} implies in particular Theorem \ref{thm:intro-PFG}.

\end{document}